\documentclass[a4paper,11pt,english]{amsart}
\usepackage[english]{babel}
\usepackage[numbers]{natbib}
\usepackage[utf8]{inputenc}
\usepackage{framed}
\usepackage[svgnames]{xcolor}
\usepackage{tikz}
\usepackage{geometry}
\usepackage[foot]{amsaddr}
\usepackage{numprint}
\usepackage[textsize=footnotesize,color=DarkBlue!40]{todonotes}

\usepackage[english]{babel}
\usepackage[labelfont=bf]{caption}
\usepackage[labelfont=bf]{subcaption}
\usepackage{comment}
\usepackage{footmisc}
\usepackage{bbm}
\usepackage{amsmath}
\usepackage{amssymb}
\usepackage{amsthm}
\usepackage{ifpdf}

\ifpdf
    \usepackage[pdftex]{hyperref}
\else
    \usepackage[hypertex]{hyperref}
\fi

\usepackage[svgnames]{xcolor}
\definecolor{hrefcolor}{rgb}{0.0,0.5,0.8}
\definecolor{hlgreen}{rgb}{0,0.7,0}

\hypersetup{
   colorlinks=true,
   linkcolor=hrefcolor,
   citecolor=hlgreen,
   filecolor=hrefcolor,
   urlcolor=hrefcolor,
}

\ifdefined\tikz
\else
    \ifpdf
        \usepackage[pdftex]{graphicx}
    \else
        \usepackage{graphicx}
        
    \fi

\fi

\newcounter{remcount}

\newtheorem{theorem}{Theorem}

\newtheorem{lemma}{Lemma}

\theoremstyle{definition}
\newtheorem{definition}{Definition}

\newtheorem*{assumption*}{Assumption}
\newtheorem{remark}{Remark}
\newtheorem*{remark*}{Remark}
\newtheorem*{definition*}{Definition}

\newtheorem{algorithm}{Algorithm}

\numberwithin{equation}{section}
\numberwithin{lemma}{section}
\numberwithin{theorem}{section}
\numberwithin{proposition}{section}
\numberwithin{definition}{section}
\numberwithin{remark}{section}
\numberwithin{example}{section}
\numberwithin{assumption}{section}
\numberwithin{algorithm}{section}
\numberwithin{corollary}{section}

\makeatletter
\AtBeginDocument{
    \hypersetup{
        pdftitle = {\@title},
        pdfauthor = {\@author}
    }
}
\makeatother

\newcommand{\field}[1]{\mathbb{#1}}

\newcommand{\R}{\field{R}}

\newcommand{\norm}[1]{\|#1\|}

\newcommand{\abs}[1]{|#1|}

\newcommand{\grad}{\nabla}
\newcommand{\freevar}{\,\boldsymbol\cdot\,}

\newcommand{\Union}\bigcup
\newcommand{\Isect}\bigcap
\newcommand{\union}\cup
\newcommand{\isect}\cap
\newcommand{\bigunion}\bigcup
\newcommand{\bigisect}\bigcap

\newcommand{\defeq}{:=}

\DeclareMathOperator*{\argmin}{arg\,min}

\DeclareMathOperator{\sign}{sgn}

\makeatletter
\def \uminus@sym{\setbox0=\hbox{$\cup$}\rlap{\hbox 
        to\wd0{\hss\raise0.5ex\hbox{$\scriptscriptstyle{-}$}\hss}}\box0}
    \def \uminus    {\mathrel{\uminus@sym}}
\makeatother

\newcommand{\mathvar}[1]{\textup{#1}}

\renewcommand{\tilde}{\widetilde}

\newcommand{\iprod}[2]{\langle #1,#2\rangle}
\DeclareMathOperator{\Sym}{Sym}

\newcommand{\BDspace}{\mathvar{BD}}

\newcommand{\BVspace}{\mathvar{BV}}

\newcommand{\Meas}{\mathcal{M}}

\renewcommand{\d}{\,d} 

\newcommand{\TGV}{\mathvar{TGV}}
\newcommand{\TV}{\mathvar{TV}}

\makeatletter
\def \weaktostar@sym{\setbox0=\hbox{$\rightharpoonup$}\rlap{\hbox 
        to\wd0{\hss\raise1ex\hbox{$\scriptscriptstyle{*\,}$}\hss}}\box0}
    \def \weaktostar    {\mathrel{\weaktostar@sym}}
\makeatother

\theoremstyle{definition}
\newtheorem{assumptionx}{Assumption}

\DeclareMathOperator{\KRTV}{KRTV}

\def\DIMurange{m}
\def\DIMdomain{n}
\def\DIMkurange{d}
\def\pixels{\ell}

\def\alphavec{\vec\alpha}

\def\ee{\text{\textsc{e}}}
\def\costltwo{{\ensuremath{L_2^2}}}

\def\costhubertv{{\ensuremath{L_\eta^1\!\nabla}}}
\def\costsymbreg{{\ensuremath{B_sL_1\!\nabla_\eta}}}
\def\NA{N}

\newcommand{\Jmu}[2][\lambda,\alpha]{J^{\gamma,\mu}(#2; #1)}

\def\DiagO{\mathfrak{D}}
\def\maxO{\mathfrak{m}}
\def\nO{\mathfrak{N}}
\def\projO{\mathfrak{P}}
\def\costf{F}

\newcommand{\huber}[2][\gamma]{\abs{#2}_{#1}}

\def\LaplaceU{\Delta}

\def\Step{\delta}

\def\SPACEuHdual{H^{1}(\Omega; \R^\DIMurange)}

\newcommand{\SPACEkuLp}[2][\DIMkurange]{L^{#2}(\Omega; \R^{#1})}

\newcommand{\SPACEkuLtwo}[1][\DIMkurange]{\SPACEkuLp[#1]{2}}

\def\SPACEuLtwo{L^2(\Omega; \R^{\DIMurange})}

\newcommand{\ICTV}{\mathvar{ICTV}}
\newcommand{\POLYTV}{\mathvar{POLYTV}}

\newlength{\colw}

\def\iftgv#1#2#3{{\def\iftgvpar{#1}\def\iftgvtgv{tgv2}\ifx\iftgvpar\iftgvtgv#2\else#3\fi}}
\def\ifkrtv#1#2#3{{\def\ifkrtvpar{#1}\def\ifkrtvkrtv{krtv2}\ifx\ifkrtvpar\ifkrtvkrtv#2\else#3\fi}}
\def\ifictv#1#2#3{{\def\ifictvpar{#1}\def\ifictvictv{ictv}\ifx\ifictvpar\ifictvictv#2\else#3\fi}}
\def\ifpolytv#1#2#3{{\def\ifpolytvpar{#1}\def\ifpolytvpolytv{polytv}\ifx\ifpolytvpar\ifpolytvpolytv#2\else#3\fi}}
\def\iftwoparam#1#2#3{\iftgv{#1}{#2}{\ifictv{#1}{#2}{#3}}}
\def\ifhuber#1#2#3{{\def\ifhpar{#1}\def\ifhh{huberonly}\ifx\ifhpar\ifhh#2\else#3\fi}}
\def\ifsymbreg#1#2#3{{\def\ifhpar{#1}\def\ifhh{symbregman}\ifx\ifhpar\ifhh#2\else#3\fi}}
\def\ifssn#1#2#3{{\def\ifhpar{#1}\def\ifhh{ssn}\ifx\ifhpar\ifhh#2\else#3\fi}}
\def\mathname#1{\iftgv{#1}{$\TGV^2$}{\ifkrtv{#1}{$\KRTV$}{\ifictv{#1}{$\ICTV$}{\ifpolytv{#1}{$\POLYTV_1$}{$\TV$}}}}}
\def\costname#1{\ifhuber{#1}{\costhubertv}{\ifsymbreg{#1}{\costsymbreg}{\costltwo}}}

\def\TVINIT{(\alpha_{\TV}^*/\pixels, \alpha_{\TV}^*)}

\newcommand{\methodnamex}[2][]{
    \ifssn{#2}{SSN#1, $\epsilon=\RESmu$, $\gamma=\RESgamma$}{Chambolle-Pock#1}
}

\def\XXint#1#2#3{{\setbox0=\hbox{$#1{#2#3}{\int}$ }
\vcenter{\hbox{$#2#3$ }}\kern-.6\wd0}}

\newcommand{\resplotxx}[2][]{
    \setlength{\colw}{\textwidth}
    \begin{tikzpicture}
        \pgftext[at=\pgfpoint{0}{0},left,bottom]{%
            \includegraphics[width=\colw]{{#2}.png}
        }
        #1
    \end{tikzpicture}%
}

\newlength{\imw}
\newcommand{\resplot}[5][]{
    \begin{subfigure}[t]{\imw}%
    \resplotxx[#1]{resimg/#2}
    \caption{\mathname{#3} denoising, \costname{#4} cost}
    \ifdefined\subfigprefix\label{\subfigprefix:#3-#4}\else\relax\fi 
    \end{subfigure}
    }

\newcommand{\inplot}[3][]{
    \begin{subfigure}[t]{\imw}%
    \resplotxx[#1]{img/#2}
    \caption{#3}
    \ifdefined\subfigprefix\label{\subfigprefix:#2}\else\relax\fi
    \end{subfigure}
    }

\def\SQl{48}
\def\SQb{160}
\def\SQr{112}
\def\SQt{224}
\newcommand{\igraph}[2]{\includegraphics[#1,bb=48 160 112 224,clip]{{resimg/#2}}}
\newlength{\scf}

\newcommand{\drawzoomarea}{\draw[line width=1.5,dashed,color=red] (\SQl\scf, \SQb\scf) rectangle (\SQr\scf, \SQt\scf);}

\newcommand{\resplotzxx}[2][]{
    \setlength{\colw}{\textwidth}
    \setlength{\scf}{0.0039\colw} 
    \begin{tikzpicture}
        \pgftext[at=\pgfpoint{0}{0},left,bottom]{%
            \includegraphics[width=\colw]{{resimg/#2}.png}
        }
        \pgftext[at=\pgfpoint{0.55\colw}{0.05\colw},left,bottom]{%
            \igraph{width=0.4\colw}{#2}
        }
        \draw[line width=1.5, color=red] (0.55\colw, 0.05\colw) rectangle (0.95\colw, 0.45\colw);
        #1
    \end{tikzpicture}%
}

\newcommand{\restabll}[6][]{%
    \input{resimg/#2-vals.tex}%
    \mathname{#3} &  
    \costname{#4} & 
    #5 & 
    \ifpolytv{#3}{%
        $(\RESalphaSCone,\RESalphaSCtwo,\RESalphaSCthree,\RESalphaSCfour)/\pixels$%
    }{%
        \iftwoparam{#3}%
              {$(\RESalphaSCtwo/\pixels^2, \RESalphaSCone/\pixels)$}%
              {$\RESalphaSCone/\pixels$}
    } &
   \RESdist & 
   \RESssim & 
   \RESpsnr &
   \RESiters &
   \def\arg{#6}
   \def\nofig{}
   \ifx\arg\nofig\else\ref{#6}(\subref{#6:#3-#4})\fi \\%
}

\newcounter{imcounter}
\newcounter{annotationcounter}

\newcommand{\immatrixX}[4]{
    \setcounter{imcounter}{1}
    \def\IMX##1##2{%
        \setcounter{annotationcounter}{0}%
        \begin{tikzpicture}%
            \pgftext[at=\pgfpoint{0}{0},left,top]{%
                \includegraphics[width=\textwidth/#3]{#1/#2_##1}%
            }%
            ##2%
        \end{tikzpicture}%
        \stepcounter{imcounter}%
        \ifnum \value{imcounter} > #3 %
            \setcounter{imcounter}{1}%
            \\%
        \fi%
    }
    \def\IM##1{\IMX{##1}{}}
    {\setlength{\lineskip}{0pt}%
    \input{#4}%
    }
}

\newcommand{\immatrix}[3]{
    \immatrixX{#1}{#2}{#3}{#1/filelist.tex}
}

\begin{document}

\title[Optimal higher-order total variation regularisation]{Bilevel parameter learning for higher-order total variation regularisation models$^*$}
\author{J.C. De los Reyes$^1$, C.-B. Sch\"onlieb$^2$ and T. Valkonen$^{2}$}
\address{$^1$Research Center on Mathematical Modelling (MODEMAT), Escuela Polit\'ecnica Nacional, Quito, Ecuador.}
\address{$^2$Department of Applied Mathematics and Theoretical Physics, University of Cambridge, Cambridge, United Kingdom.}
\thanks{$^*$This research has been supported by King Abdullah University of Science and Technology (KAUST) Award No.~KUK-I1-007-43, EPSRC grants Nr.~EP/J009539/1 ``Sparse \& Higher-order Image Restoration'' and Nr.~EP/M00483X/1 ``Efficient computational tools for inverse imaging problems'', Escuela Politécnica Nacional de Quito Award No. PIS 12-14 and MATHAmSud project SOCDE ``Sparse Optimal Control of Differential Equations''. While in Quito, T.~Valkonen has moreover been supported by SENESCYT (Ecuadorian Ministry of Higher Education, Science, Technology and Innovation) under a Prometeo Fellowship.}

\maketitle

\begin{abstract}
We consider a bilevel optimisation approach for parameter learning in higher-order total variation image reconstruction models. Apart from the least squares cost functional, naturally used in bilevel learning, we propose and analyse an alternative cost, based on a Huber regularised TV-seminorm. 
Differentiability properties of the solution operator are verified and a first-order optimality system is derived. Based on the adjoint information, a quasi-Newton algorithm is proposed for the numerical solution of the bilevel problems. Numerical experiments are carried out to show the suitability of our approach and the improved performance of the new cost functional. Thanks to the bilevel optimisation framework, also a detailed comparison between $\TGV^2$ and $\ICTV$ is carried out, showing the advantages and shortcomings of both regularisers, depending on the structure of the processed images and their noise level.
\end{abstract}

\section{Introduction}

In this paper we propose a bilevel optimisation approach for parameter learning in higher-order total variation regularisation models for image restoration. The reconstruction of an image from imperfect measurements is essential for all research which relies on the analysis and interpretation of image content. Mathematical image reconstruction approaches aim to maximise the information gain from acquired image data by intelligent modelling and mathematical analysis.


A variational image reconstruction model can be formalised as follows. Given data $f$ which is related to an image (or to certain image information, e.g. a segmented or edge detected image) $u$ through a generic forward operator (or function) $K$ the task is to retrieve $u$ from $f$. In most realistic situations this retrieval is complicated by the ill-posedness of $K$ as well as random noise in $f$. A widely accepted method that approximates this ill-posed problem by a well-posed one and counteracts the noise is the method of Tikhonov regularisation. That is, an approximation to the true image is computed as a minimiser of 
\begin{equation}\label{reg1:eq}
\alpha~ R(u) + d(K(u),f),
\end{equation}
where $R$ is a regularising energy that models a-priori knowledge about the image $u$, $d(\cdot , \cdot)$ is a suitable distance function that models the relation of the data $f$ to the unknown $u$, and $\alpha>0$ is a parameter that balances our trust in the forward model against the need of regularisation. The parameter $\alpha$ in particular, depends on the amount of ill-posedness in the operator $K$ and the amount (amplitude) of the noise present in $f$. A key issue in imaging inverse problems is the correct choice of $\alpha$, image priors (regularisation functionals $R$), fidelity terms $d$ and (if applicable) the choice of what to measure (the linear or nonlinear operator $K$). Depending on this choice, different reconstruction results are obtained. 


While functional modelling \eqref{reg1:eq} constitutes a mathematically rigorous and physical way of setting up the reconstruction of an image -- providing reconstruction guarantees in terms of error and stability estimates -- it is limited with respect to its adaptivity for real data. On the other hand, data-based modelling of reconstruction approaches is set up to produce results which are optimal with respect to the given data. However, in general it neither offers insights into the structural properties of the model nor provides comprehensible reconstruction guarantees. Indeed, we believe that for the development of reliable, comprehensible and at the same time effective models \eqref{reg1:eq} it is essential to aim for a unified approach that seeks tailor-made regularisation and data models by combining model- and data-based approaches.


%

To do so we focus on a bilevel optimisation strategy for finding an optimal setup of variational regularisation models \eqref{reg1:eq}. That is, for a given training pair of noisy and original clean images $(f,f_0)$, respectively, we consider a learning problem of the form 
\begin{equation}\label{eq:generprob}
\min F(u^*) = cost(u^*,f_0) \quad \textrm{subject to} \quad u^*\in \argmin_u \left\{\alpha~ R(u) + d(K(u),f)\right\},
\end{equation}
where $F$ is a generic cost functional that measures the fitness of $u^*$ to the original image $f_0$. The argument of the minimisation problem will depend on the specific setup (i.e. the degrees of freedom) in the constraint problem \eqref{reg1:eq}. In particular, we propose a bilevel optimisation approach for learning optimal parameters in higher-order total variation regularisation models for image reconstruction in which the arguments of the optimisation constitute parameters in front of the first- and higher-order regularisation terms. Rather than working on the discrete problem, as is done in standard parameter learning and model optimisation methods, we optimise the regularisation models in infinite dimensional function space. We will explain this approach in more detail in the next section. Before, let us give an account to the state of the art of bilevel optimisation for model learning. In machine learning bilevel optimisation is well established. It is a semi-supervised learning method that optimally adapts itself to a given dataset of measurements and desirable solutions. In \cite{tappen2007utilizing,domke2012generic,chen2014}, for instance the authors consider bilevel optimization for finite dimensional Markov random field models. In inverse problems the optimal inversion and experimental acquisition setup is discussed in the context of optimal model design in works by Haber, Horesh and Tenorio \cite{haber2003learning,haber2010numerical}, as well as Ghattas et al. \cite{bui2008model,biegler2011large}.  Recently parameter learning in the context of functional variational regularisation models \eqref{reg1:eq} also entered the image processing community with works by the authors \cite{de2013image,calatronidynamic}, Kunisch, Pock and co-workers \cite{kunisch2013bilevel,Chen2012}, Chung et al. \cite{chung2014optimal} and Hinterm\"uller et al. \cite{hintermuller2014bilevel}. 

Apart from the work of the authors \cite{de2013image,calatronidynamic}, all approaches so far are formulated and optimised in the discrete setting. Our subsequent modelling, analysis and optimisation will be carried out in function space rather than on a discretisation of \eqref{reg1:eq}. While digitally acquired image data is of course discrete, the aim of high resolution image reconstruction and processing is always to compute an image that is close to the real (analogue, infinite dimensional) world. Hence, it makes sense to seek images which have certain properties in an infinite dimensional function space. That is, we aim for a processing method that accentuates and preserves qualitative properties in images independent of the resolution of the image itself \cite{viola2012unifying}. Moreover, optimisation methods conceived in function space potentially result in numerical iterative schemes which are resolution and mesh-independent upon discretisation \cite{hintermuller2006infeasible}.

Higher-order total variation regularisation has been introduced as an extension of the standard total variation regulariser in image processing. As the {\bf T}otal {\bf V}ariation (TV) \cite{Rud1992} and many more contributions in the image processing community have proven, a non-smooth first-order regularisation procedure results in a nonlinear smoothing of the image, smoothing more in homogeneous areas of the image domain and preserving characteristic structures such as edges. In particular, the TV regulariser is tuned towards the preservation of edges and performs very well if the reconstructed image is piecewise constant. The drawback of such a regularisation procedure becomes apparent as soon as images or signals (in 1D) are considered which do not only consist of constant regions and jumps, but also possess more complicated, higher-order structures, e.g. piecewise linear parts. The artefact introduced by TV regularisation in this case is called staircasing \cite{ring2000structural}. One possibility to counteract such artefacts is the introduction of higher-order derivatives in the image regularisation. Chambolle and Lions \cite{chambolle97image}, for instance, propose a higher order method by means of an infimal convolution of the TV and the TV of the image gradient called {\bf I}nfimal-{\bf C}onvolution {\bf T}otal {\bf V}ariation (ICTV) model. Other approaches to combine first and second order regularisation originate, for instance, from Chan, Marquina, and Mulet \cite{chan2000highorder} who consider total variation minimisation together with weighted versions of the Laplacian, the Euler-elastica functional \cite{masnou1998level,chan2002euler} which combines total variation regularization with curvature penalisation, and many more \cite{lysaker2006iterative,papafitsoros2012combined} just to name a few. Recently Bredies et al. have proposed {\bf T}otal {\bf G}eneralized {\bf V}ariation (TGV) \cite{bredies2011tgv} as a higher-order variant of TV regularisation.

In this work we mainly concentrate on two second-order total variation models: the recently proposed TGV \cite{bredies2011tgv} and the ICTV model of Chambolle and Lions \cite{chambolle97image}. We focus on second-order TV regularisation only since this is the one which seems to be most relevant in imaging applications \cite{knoll2010second,bredies2012total}. For $\Omega\subset\mathbb R^2$ open and bounded and $u\in BV(\Omega)$, the ICTV regulariser reads
\begin{equation}\label{eq:ictv}
\ICTV_{\alpha,\beta}(u)\defeq \min_{v\in W^{1,1}(\Omega),~ \nabla v\in BV(\Omega)} \alpha \norm{Du-\nabla v}_{\Meas(\Omega; \R^2)} + \beta \norm{D\nabla v}_{\Meas(\Omega; \R^{2\times 2})}.
\end{equation}
On the other hand, second-order TGV \cite{sampta2011tgv,l1tgv} for $u\in BV(\Omega)$ reads
\begin{equation}
    \label{eq:tgv}
    \TGV^2_{\alpha,\beta}(u) 
    \defeq
    \min_{w \in BD(\Omega)}
    \alpha \norm{Du-w}_{\Meas(\Omega; \R^2)}
    +
    \beta \norm{Ew}_{\Meas(\Omega; \Sym^2(\R^2))}.
\end{equation}
Here $
    \BDspace(\Omega)
    \defeq
    \{ w \in L^1(\Omega; \R^\DIMdomain)
        \mid
        \norm{Ew}_{\Meas(\Omega; \R^{\DIMdomain \times \DIMdomain})}
     < \infty \}
$
is the space of vector fields of bounded deformation on $\Omega$, $E$ denotes the \emph{symmetrised gradient} and $\mathrm{Sym}^2(\mathbb{R}^2)$ the space of symmetric tensors of order $2$ with arguments in $\mathbb{R}^2$. The parameters $\alpha,\beta$ are fixed positive parameters and will constitute the arguments in the special learning problem \'a la \eqref{eq:generprob} we consider in this paper. The main difference between \eqref{eq:ictv} and \eqref{eq:tgv} is that we do not generally have that $w=\nabla v$ for any function $v$. That results in some qualitative differences of ICTV and TGV regularisation, compare for instance \cite{benning2011higher}. Substituting $\alpha R(u)$ in \eqref{reg1:eq} by $\TGV^2_{\alpha,\beta}(u)$ or $\ICTV_{\alpha,\beta}(u)$ gives the TGV image reconstruction model and the ICTV image reconstruction model, respectively. In this paper we only consider the case $K=Id$ identity and $d(u,f)=\|u-f\|_{L^2(\Omega)}^2$ in \eqref{reg1:eq} which corresponds to an image de-noising model for removing Gaussian noise. With our choice of regulariser the former scalar $\alpha$ in \eqref{reg1:eq} has been replaced by a vector $(\alpha,\beta)$ of two parameters in \eqref{eq:tgv} and \eqref{eq:ictv}. The choice of the entries in this vector now do not only determine the overall strength of the regularisation (depending on the properties of $K$ and the noise level) but those parameters also balance between the different orders of regularity of the function $u$, and their choice is indeed crucial for the image reconstruction result. Large $\beta$ will give regularised solutions that are close to TV regularised reconstructions, compare Figure \ref{fig:intro}. Large $\alpha$ will result in TV$^2$ type solutions, that is solutions that are regularised with TV of the gradient \cite{hinterberger2006variational,papafitsoros2012combined}, compare Figure \ref{fig:intro2}.  With our approach described in the next section we propose a learning approach for choosing those parameters optimally, in particular optimally for particular types of images. 

\begingroup
    \def\SQl{128}
    \def\SQb{160}
    \def\SQr{192}
    \def\SQt{224}
    \renewcommand{\igraph}[2]{\includegraphics[#1,bb=128 160 192 224,clip]{{resimg/#2}.png}}

    \begin{figure}[h!]
        \centering
        \begin{subfigure}[t]{0.3\textwidth}%
            \resplotzxx{intro-low}
            \caption{Too low $\beta$ / High oscillation}
        \end{subfigure}
        \begin{subfigure}[t]{0.3\textwidth}%
            \resplotzxx[\drawzoomarea]{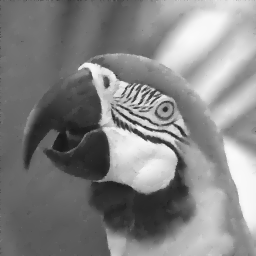}
            \caption{Optimal $\beta$}
        \end{subfigure}
        \begin{subfigure}[t]{0.3\textwidth}%
            \resplotzxx{intro-high}
            \caption{Too high $\beta$ / almost $\TV$}
        \end{subfigure}
        \caption{Effect of $\beta$ on $\TGV^2$ denoising with optimal $\alpha$}
        \label{fig:intro}
    \end{figure}

    \begin{figure}[h!]
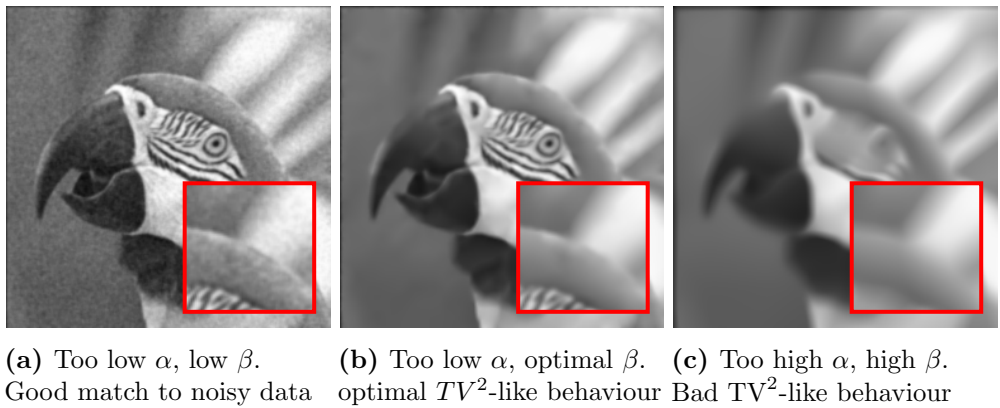

        \centering
        \begin{subfigure}[t]{0.3\textwidth}%
            \resplotzxx{intro-alphahigh-betalow}
            \caption{Too low $\alpha$, low $\beta$. \\Good match to noisy data}
        \end{subfigure}
        \begin{subfigure}[t]{0.3\textwidth}%
            \resplotzxx{intro-alphahigh-betaopt}
            \caption{Too low $\alpha$, optimal $\beta$. \\ optimal $TV^2$-like behaviour}
        \end{subfigure}
        \begin{subfigure}[t]{0.3\textwidth}%
            \resplotzxx{intro-alphahigh-betahigh}
            \caption{Too high $\alpha$, high $\beta$.\\ Bad $\TV^2$-like behaviour}
        \end{subfigure}
        \caption{Effect of choosing $\alpha$ too large in $\TGV^2$ denoising}
        \label{fig:intro2}
    \end{figure}
\endgroup

For the existence analysis of an optimal solution as well as for the derivation of an optimality system for the corresponding learning problem \eqref{eq:generprob} we will consider a smoothed version of the constraint problem \eqref{reg1:eq} -- which is the one in fact used in the numerics. That is, we replace $R(u)$ -- being TV, TGV or ICTV in this paper -- by a Huber regularised version and add an $H^1$ regularisation with a small weight to \eqref{reg1:eq}. In this setting and under the special assumption of box constraints on $\alpha$ and $\beta$ we provide a simple existence proof for an optimal solution. A more general existence result that holds also for the original non-smooth problem and does not require box constraints is derived in \cite{tuomov-interior} and we refer the reader to this paper for a more sophisticated analysis on the structure of solutions.


A main challenge in the setup of such a learning approach is to decide what is the best way to measure fitness (optimality) of the model. In our setting this amounts to choosing an appropriate distance $F$ in \eqref{eq:generprob} that measures the fitness of reconstructed images to the `perfect', noise-free images in an appropriate training set. We have to formalise what we mean by an optimal reconstruction model. Classically, the difference between the original, noise-free image $f_0$ and its regularised version $u_{\alpha,\beta}$ is computed with an $\costltwo$ cost functional
\begin{equation}\label{eq:costL2}
\costf_{\costltwo}(u_{\alpha,\beta}) = \norm{u_{\alpha,\beta} - f_0}_{L^2(\Omega)}^2,
\end{equation}
which is closely related to the PSNR quality measure. Apart from this, we propose in this paper an alternative cost functional based on a Huberised total variation cost
\begin{equation}\label{eq:costTVhub}
\costf_{\costhubertv}(u_{\alpha,\beta}) \defeq \int_\Omega \huber[\gamma]{D(u_{\alpha,\beta}-f_0)}~ dx,
\end{equation}
where the Huber regularisation $\huber[\gamma]{\cdot}$ will be defined later on in Definition \ref{def:huber}. We will see that the choice of this cost functional is indeed crucial for the qualitative properties of the reconstructed image.

The proposed bilevel approach has an important indirect consequence: It establishes a basis for the comparison of the different total variation regularisers employed in image denoising tasks. In the last part of the paper we exhaustively compare the performance of $\TV$, $\TGV^2$ and $\ICTV$ for various image datasets. The parameters are chosen optimally, according to the proposed bilevel approach, and different quality measures (like PSNR and SSIM) are considered for the comparison. The obtained results are enlightening about when to use each one of the considered regularisers. In particular, $\ICTV$ appears to behave better for images with arbitrary structure and moderate noise levels, whereas $\TGV^2$ behaves better for images with large smooth areas.\\


\noindent{\bf Outline of the paper} In Section \ref{sec:probstatex} we state the bilevel learning problem for the two higher-order total variation regularisation models, TGV and ICTV, and prove existence of an optimal parameter pair $\alpha,\beta$. The bilevel optimization problem is analysed in Section \ref{sec:Lagrange}, where existence of Lagrange multipliers is proved and an optimality system, as well as a gradient formula, are derived. Based on the optimality condition, a BFGS algorithm for the bilevel learning problem is devised in Section \ref{sec:denoising}. For the numerical solution of each denoising problem an infeasible semi-smooth Newton method is considered. Finally, we discuss the performance of the parameter learning method by means of several examples for the denoising of natural photographs in Section \ref{sec:experiments}. Therein, we also present a statistical analysis on how TV, ICTV and TGV regularisation compare in terms of returned image quality, carried out on 200 images from the Berkeley segmentation dataset BSDS300.

\section{Problem statement and existence analysis}\label{sec:probstatex}

We strive to develop a parameter learning method for higher-order total variation regularisation models that maximises the fit of the reconstructed images to training images simulated for an application at hand. For a given noisy image $f\in L^2(\Omega)$, $\Omega\subset\mathbb R^2$ open and bounded, we consider
\begin{equation}\label{tgvdenoise}
    \min_u \left\{R_{\alpha, \beta}(u) + \frac{1}{2} \|u-f\|_{L^2(\Omega)}^2\right\}.
\end{equation}
where, $\alpha,\beta\in\mathbb R$. We focus on TGV$^2$ and ICTV image denoising:
\begin{multline*}
R_{\alpha,\beta}(u)=\TGV^2_{\alpha,\beta}(u)\defeq
    \min_{w \in BD(\Omega)}
    \norm{\alpha~(Du-w)}_{\Meas(\Omega; \R^2)}
    \\+
    \norm{\beta~Ew}_{\Meas(\Omega; \Sym^2(\R^2))}.
\end{multline*}
and \eqref{eq:ictv} with spatial dependence
\begin{multline*}
R_{\alpha,\beta}(u)=\ICTV_{\alpha, \beta}(u)\defeq\min_{\substack{v\in W^{1,1}(\Omega)\\ \nabla v\in BV(\Omega)}} \norm{\alpha~(Du-\nabla v)}_{\Meas(\Omega; \R^2)} \\
+ \norm{\beta~D\nabla v}_{\Meas(\Omega; \R^{2\times 2})},
\end{multline*}
for $u\in BV(\Omega)$. For this model, we want to determine the optimal choice of $\alpha,\beta$, given a particular type of images and a fixed noise level. More precisely, we consider a training pair $(f,f_0)$, where $f$ is a noisy image corrupted by normally distributed noise with a fixed variation, and the image $f_0$ represents the ground truth or an image that approximates the ground truth within a desirable tolerance. Then, we determine the optimal choice of $\alpha,\beta$ by solving the following problem
\begin{equation}\label{optimalcontrolbasic}
    \min_{(\alpha,\beta)\in\mathbb R^{2}} ~  \costf(u_{\alpha,\beta}) \quad \textrm{ s.t. } \alpha, \beta \geq 0,
\end{equation}
where $\costf$ equals the $\costltwo$ cost \eqref{eq:costL2} or the Huberised TV cost \eqref{eq:costTVhub} and $u_{\alpha,\beta}$ for a given $f$ solves a regularised version of the minimization problem \eqref{tgvdenoise} that will be specified in the next section, compare problem \eqref{eq:denoise-huber1}. This regularisation of the problem is a technical requirement for solving the bilevel problem that will be discussed in the sequel. In contrast to learning $\alpha,\beta$ in \eqref{tgvdenoise} in finite dimensional parameter spaces (as is the case in machine learning) we aim for novel optimisation techniques in infinite dimensional function spaces.

\subsection{Formal statement}

Let $\Omega \subset \R^\DIMdomain$ be an open bounded domain with Lipschitz boundary. This will be our image domain. Usually $\Omega=(0, w) \times (0, h)$ for $w$ and $h$ the width and height of a two-dimensional image, although no such assumptions are made in this work. Our data $f$ and $f_0$ are assumed to lie in $L^2(\Omega)$. 

In our learning problem, we look for parameters $(\alpha,\beta)$ that for some cost functional $F: H^1(\Omega) \to \R$ solve the problem
\begin{subequations} \label{eq:learn}
\begin{equation}
   \min_{(\alpha,\beta)\in\mathbb R^{2}} ~  \costf(u_{\alpha,\beta})
\end{equation}
subject to
\begin{align}
    & u_{\alpha,\beta} \in \argmin_{u\in H^1(\Omega)} \Jmu[\alpha,\beta]{u} \label{eq:denoise-huber1} \\ 
    & \alpha, \beta \geq 0,
\end{align}
\end{subequations}
where
\[
    \Jmu[\alpha,\beta]{u} \defeq
       \frac12 \norm{u-f}_{L^2(\Omega)}^2
        + R_{\alpha,\beta}^{\gamma,\mu}(u).
\]
Here $\Jmu[\alpha,\beta]{\cdot}$ is the regularised denoising functional that amends the regularisation term in \eqref{tgvdenoise} by a Huber regularised version of it with parameter $\gamma>0$, and an elliptic regularisation term with parameter $\mu>0$. In the case of TGV$^2$ the modified regularisation term $R_{\alpha,\beta}^{\gamma,\mu}(u)$ then reads for $u\in H^1(\Omega)$
\begin{multline*}
\TGV^{2,\gamma,\mu}_{\alpha,\beta}(u)\defeq
    \min_{w \in H^1(\Omega)}
    \int_\Omega \alpha~|Du-w|_\gamma~ dx
   \\ +
    \int_\Omega \beta~|Ew|_\gamma~ dx + \frac{\mu}{2} \left(\norm{u}_{H^1(\Omega)}^2 + \norm{w}_{\mathbb H^1(\Omega)}^2\right)
\end{multline*}
and in the case of ICTV we have
\begin{multline*}
\ICTV_{\alpha, \beta}^{\gamma,\mu}(u)\defeq\min_{\substack{v\in W^{1,1}(\Omega)\\\nabla v\in BV(\Omega,\mathbb R^n)\cap \mathbb H^1(\Omega)}} \int_\Omega \alpha ~|Du-\nabla v|_\gamma~dx \\
+ \int_\Omega \beta~|D\nabla v|_\gamma~ dx + \frac{\mu}{2} \left(\norm{u}_{H^1(\Omega)}^2 + \norm{\nabla v}_{\mathbb H^1(\Omega)}^2\right).
\end{multline*}
Here, $\mathbb H^1(\Omega)=H^1(\Omega;\mathbb R^n)$ and the Huber regularisation $|\cdot|_\gamma$ is defined as follows.

\begin{definition}\label{def:huber}
Given $\gamma \in (0, \infty]$, we  define for the norm $\norm{\freevar}_2$ on $\R^m$, 
the Huber regularisation
\[
    \huber[\gamma]{g} = 
    \begin{cases}
        \norm{g}_2 - \frac{1}{2\gamma}, & \norm{g}_2 \ge 1/\gamma,
        \\
        \frac{\gamma}{2}\norm{g}_2^2, & \norm{g}_2 < 1/\gamma.
    \end{cases}
\]
\end{definition}

For the cost functional $\costf$, given noise-free data $f_0 \in L^2(\Omega)$ and a regularised solution $u\in H^1(\Omega)$, we consider in particular the $L^2$ cost
\[
    \costf_{\costltwo}(u) \defeq \frac{1}{2}\norm{f_0 - u}_{\SPACEkuLtwo}^2, 
\]
as well as the Huberised total variation cost
\[
    \costf_{\costhubertv}(u) \defeq \int_\Omega\huber[\gamma]{D(f_0-u)}~ dx
\]
with noise-free data $f_0 \in \BVspace(\Omega)$.

\subsection{Existence of an optimal solution}
The existence of an optimal solution for the learning problem \eqref{eq:learn} is a special case of the class of bilevel problems considered in \cite{tuomov-interior}, where existence of optimal parameters in $(0,+\infty]^{2N}$ is proven. For convenience, we provide a simplified proof for the case where box constraints on the parameters are imposed. We start with an auxiliary lower semicontinuity result for the Huber regularised functionals.
\begin{lemma}\label{lem:huberlsc}
Let $u,v\in L^p(\Omega)$, $1\leq p<\infty$. Then, the functional $u \mapsto \int_\Omega |u-v|_\gamma~ dx$, where $|\cdot|_\gamma$ is the Huber regularisation in Definition \ref{def:huber}, is lower semicontinuous with respect to weak* convergence in $\Meas(\Omega; \R^\DIMkurange)$
\end{lemma}
\begin{proof}
Recall that for $g \in \R^m$, the Huber-regularised norm may be written in dual form as
    \begin{equation}
        \notag
        \huber[\gamma]{g}=
        \sup
        \Bigl\{
            \iprod{q}{g} - \frac{\gamma}{2} \norm{q}_2^2
            :
            \norm{q}_2 \le 1
        \Bigr\}.
    \end{equation}
    Therefore, we find that
\begin{multline*}
        G(u)\defeq\int_\Omega |u-v|_\gamma~ dx=\sup\Big\{
             \int_\Omega u(x)\cdot\varphi(x)~ dx - \int_\Omega \frac{\gamma}{2} \norm{\varphi(x)}_2^2 \d x :
             \\
                 \varphi \in C_c^\infty(\Omega),\ 
                 \norm{\varphi(x)}_2 \le \alpha \text{ for every } x \in \Omega
             \Big\}.
\end{multline*}
The functional $G$ is of the form $G(u) = \sup\{\iprod{u}{\varphi}-G^*(\varphi)\}$, where $G^*$ is the convex conjugate of $G$. Now, let $\{u^i\}_{i=1}^\infty$ converge to $u$ weakly* in $\Meas(\Omega; \R^\DIMkurange)$. Taking a supremising sequence $\{\varphi^j\}_{j=1}^\infty$ for
    this functional at any point $u$, we easily see lower semicontinuity
    by considering the sequences $\{\iprod{u^i}{\varphi^j}-G^*(\varphi^j)\}_{i=1}^\infty$
    for each $j$. 
\end{proof}

Our main existence result is the following.
\begin{theorem}
We consider the learning problem \eqref{eq:learn} for TGV$^2$ and ICTV regularisation, optimising over parameters $(\alpha,\beta)$ such that $0 \leq \alpha \leq \bar \alpha, 0 \leq \beta \leq \bar \beta$. Here $(\bar\alpha,\bar\beta)<\infty$ is an arbitrary but fixed vector in $\mathbb R^{2}$ that defines a box constraint on the parameter space. Then, there exists an optimal solution $(\hat\alpha,\hat\beta)\in\mathbb R^{2}$ for this problem for both choices of cost functionals, $\costf=\costltwo$ and $\costf=\costf_{\costhubertv}$.
\end{theorem}

\begin{proof}
Let $(\alpha_n,\beta_n)\subset \mathbb R^{2}$ be a minimising sequence. Due to the box constraints we have that the sequence $(\alpha_n,\beta_n)$ is bounded in $ \mathbb R^{2}$. Moreover, we get for the corresponding sequences of states $u_n:= u_{(\alpha_n,\beta_n)}$ that
$$
\Jmu[\alpha_n,\beta_n]{u_n} \leq \Jmu[\alpha_n,\beta_n]{u}, \quad \forall u\in H^1(\Omega),
$$
in particular this holds for $u=0$. Hence,
\begin{equation}\label{eq:compbound}
 \frac12 \norm{u_n-f}_{L^2(\Omega)}^2
        + R_{\alpha_n,\beta_n}^{\gamma,\mu}(u_n) \leq \frac12 \norm{f}_{L^2(\Omega)}^2.
\end{equation}
Exemplarily, we consider here the case for the TGV regulariser, that is $R_{\alpha_n,\beta_n}^{\gamma,\mu} = \TGV^{2,\gamma,\mu}_{\alpha,\beta}$. The proof for the ICTV regulariser can be done in a similar fashion. Inequality \eqref{eq:compbound} in particular gives
$$
\norm{u_n}_{H^1(\Omega)}^2 + \norm{w_n}_{\mathbb H^1(\Omega)}^2 \leq \frac{1}{\mu} \norm{f}_{L^2(\Omega)},
$$
where $w_n$ is the optimal $w$ for $u_n$. This gives that $(u_n,w_n)$ is uniformly bounded in $H^1(\Omega)\times \mathbb H^1(\Omega)$ and that there exists a subsequence $\{(\alpha_n,\beta_n,u_n,w_n)\}$ which converges weakly in $\mathbb R^{2}\times H^1(\Omega)\times \mathbb H^1(\Omega)$ to a limit point $(\hat\alpha,\hat\beta,\hat u,\hat w)$. 
Moreover, $u_n\rightarrow \hat u$ strongly in $L^p(\Omega)$ and $w_n\rightarrow \hat w$ in $L^p(\Omega;\mathbb R^n)$. Using the continuity of the $L^2$ fidelity term with respect to strong convergence in $L^2$, and the weak lower semicontinuity of the $H^1$ term with respect to weak convergence in $H^1$ and of the Huber regularised functional even with respect to weak$*$ convergence in $\mathcal M$ (cf. Lemma \ref{lem:huberlsc}) we get
\begin{align*}
& & & \frac12 \norm{\hat u-f}_{L^2(\Omega)}^2
    + \int_\Omega \hat \alpha~|D\hat u-\hat w|_\gamma~ dx +
    \int_\Omega \hat\beta~|Ew|_\gamma~ dx \\
&  & & + \frac{\mu}{2} \left(\norm{\hat u}_{H^1(\Omega)}^2 + \norm{\hat w}_{\mathbb H^1(\Omega)}^2\right)
\\
       & \leq & & \liminf_n \frac12 \norm{u_n-f}_{L^2(\Omega)}^2
        + \int_\Omega \hat\alpha~|Du_n-w_n|_\gamma~ dx +
    \int_\Omega \hat\beta~|Ew_n|_\gamma~ dx 
   \\ & & &  + \frac{\mu}{2} \left(\norm{u_n}_{H^1(\Omega)}^2 + \norm{w_n}_{\mathbb H^1(\Omega)}^2\right)
\\
	 & \leq & &\liminf_n \frac12 \norm{u_n-f}_{L^2(\Omega)}^2
        + \int_\Omega \alpha_n~|Du_n-w_n|_\gamma~ dx +
    \int_\Omega \beta_n~|Ew_n|_\gamma~ dx 
   \\ & & & + \frac{\mu}{2} \left(\norm{u_n}_{H^1(\Omega)}^2 + \norm{w_n}_{\mathbb H^1(\Omega)}^2\right),       
\end{align*}
where in the last step we have used the boundedness of the sequence $R_{\alpha_n,\beta_n}^{\gamma,\mu}(u_n)$ from \eqref{eq:compbound} and the convergence of $(\alpha_n,\beta_n)$ in $\mathbb R^{2}$. This shows that the limit point $\hat u$ is an optimal solution for $(\hat\alpha,\hat\beta)$. Moreover, due to the weak lower semicontinuity of the cost functional $F$ and the fact that the set $\{(\alpha,\beta):~ 0 \leq \alpha \leq \bar\alpha,0 \leq \beta \leq \bar\beta \}$ is closed, we have that $(\hat\alpha,\hat\beta,\hat u)$ is optimal for \eqref{eq:learn}.
\end{proof}

\begin{remark} \label{rem: existence}
\hspace{1cm}
\begin{itemize}
\item Using the existence result in \cite{tuomov-interior}, in principle we could allow infinite values for $\alpha$ and $\beta$. This would include both $\TV^2$ and $\TV$ as possible optimal regularisers in our learning problem. 
\item In \cite{tuomov-interior}, in the case of the $L^2$ cost and assuming that $$R_{\alpha,\beta}^{\gamma}(f)>R_{\alpha,\beta}^{\gamma}(f_0),$$ we moreover show that the parameters $(\alpha,\beta)$ are strictly larger than $0$. In the case of the Huberised TV cost this can only be proven in a discretised setting. Please see \cite{tuomov-interior} for details. 
\item The existence of solutions with $\mu=0$, that is without elliptic regularisation, is also proven in \cite{tuomov-interior}. Note that here, we focus on the $\mu>0$ case since the elliptic regularity is required for proving the existence of Lagrange multipliers in the next section.
\end{itemize}
\end{remark}

\section{Lagrange multipliers} \label{sec:Lagrange}
In this section we prove the existence of Lagrange multipliers for the learning problem \eqref{eq:learn} and derive an optimality system that characterizes its solution. Moreover, a gradient formula for the reduced cost functional is obtained, which plays an important role in the development of fast solution algorithms for the learning problems (see Section \ref{sec:denoising}).

In what follows all proofs are presented for the $\TGV^2$ regularisation case, that is $R_{\alpha,\beta}^{\gamma}=\TGV^{2,\gamma}_{\alpha,\beta}$. However, possible modifications to cope with the ICTV model will also be commented.

We start by investigating the differentiability of the solution operator.
\subsection{Differentiability of the solution operator}\label{sec:diff}
We recall that the $\TGV^2$ denoising problem is given by
\begin{equation*}
u=(v,w)= \argmin_{BV \times BD} \left\{ \frac{1}{2}\int_\Omega |v-f|^2 +  \int_\Omega \alpha |Dv-w|_\gamma + \int_\Omega \beta |E w|_\gamma  \right\}.
\end{equation*}
Using an elliptic regularization we then get
\begin{equation*} \label{eq: denoising problem in diff proof}
u= \argmin_{H^1(\Omega) \times \mathbb H^1(\Omega)} \left\{ a(u,u)+ \frac{1}{2} \int_\Omega |v-f|^2  + \int_\Omega \alpha |Dv-w|_\gamma + \int_\Omega \beta |E w|_\gamma \right\},
\end{equation*}
where $a(u,u)= \mu \left( \|v\|_{H^1}^2 + \|w\|_{\mathbb H^1}^2 \right)$. A necessary and sufficient optimality condition for the latter is then given by the following variational equation
\begin{multline} \label{eq: denoising problem in diff proof var form}
a(u, \Psi)+ \int_\Omega \alpha h_\gamma(Dv-w)(D \phi- \varphi) \,dx\\ + \int_\Omega \beta h_\gamma(E w) E \varphi \,dx +\int_\Omega(v-f)\phi \,dx=0, \text{ for all } \Psi \in U,
\end{multline}
where $\Psi=(\phi,\varphi)$ and $U=H^1(\Omega) \times \mathbb H^1(\Omega)$.

\begin{theorem}
The solution operator $S: \mathbb R^2 \mapsto U$, which assigns to each pair $(\alpha, \beta) \in \mathbb R^{2}$ the corresponding solution to the denoising problem \eqref{eq: denoising problem in diff proof var form}, is Fr\'echet differentiable and its derivative is characterized by the unique solution $z=S'(\alpha, \beta)[\theta_1, \theta_2] \in U$ of the following linearized equation:
\begin{multline} \label{eq: linearized equation}
a(z, \Psi)+ \int_\Omega \theta_1 \ h_\gamma(Dv-w)(D \phi- \varphi) \,dx\\ + \int_\Omega \alpha h'_\gamma(Dv-w)(Dz_1-z_2)(D \phi- \varphi) \,dx + \int_\Omega \theta_2 \ h_\gamma(E w) E \varphi \,dx\\+ \int_\Omega \beta h'_\gamma(E w) E z_2 E \varphi \,dx + \int_\Omega z_1 \phi \,dx=0, \text{ for all } \Psi \in U.
\end{multline} 
\end{theorem}
\begin{proof}
Thanks to the ellipticity of $a(\cdot, \cdot)$ and the monotonicity of $h_\gamma$, existence of a unique solution to the linearized equation follows from the Lax-Milgram theorem.

Let $\xi:=u^+- u -z$, where $u=S(\alpha, \beta)$ and $u^+=S(\alpha+\theta_1, \beta+ \theta_2)$. Our aim is to prove that
$\| \xi \|_U= o(|\theta|).$ Combining the equations for $u^+$, $u$ and $z$ we get that 
\begin{multline*}
a(\xi, \Psi)+ \int_\Omega (\alpha +\theta_1) \ h_\gamma(Dv^+-w^+)(D \phi- \varphi) \,dx- \int_\Omega \alpha \ h_\gamma(Dv-w)(D \phi- \varphi) \,dx\\ 
- \int_\Omega \theta_1 \ h_\gamma(Dv-w)(D \phi- \varphi) \,dx
- \int_\Omega \alpha h'_\gamma(Dv-w)(Dz_1-z_2)(D \phi- \varphi) \,dx\\ 
+ \int_\Omega (\beta+\theta_2) h_\gamma(E w^+) E \varphi \,dx- \int_\Omega \beta h_\gamma(E w) E \varphi \,dx\\ 
- \int_\Omega \theta_2 \ h_\gamma(E w) E \varphi \,dx
- \int_\Omega \beta \ h'_\gamma(E w) E z_2 E \varphi \,dx +2 \int_\Omega \xi_1 \phi \,dx=0, \text{ for all } \Psi \in U,
\end{multline*}
where $\xi:=(\xi_1,\xi_2) \in H^1(\Omega) \times \mathbb H^1(\Omega)$. Adding and subtracting the terms $$\int_\Omega \alpha h'_\gamma(Dv-w)(D \delta_v - \delta_w)(D \phi- \varphi) \,dx$$ and $$\int_\Omega \beta h'_\gamma(E w)E \delta_w: E \varphi \,dx,$$ where $\delta_v:=v_{\alpha+\theta}-v$ and $\delta_w:=w_{\alpha+\theta}-w$, we obtain that
\begin{multline*}
a(\xi, \Psi)+ \int_\Omega \alpha h'_\gamma(Dv-w)(D \xi_1- \xi_2)(D \phi- \varphi)\\+ \int_\Omega \beta h'_\gamma(E w) E \xi_2 :E \varphi \,dx +2 \int_\Omega \xi_1 \phi \,dx\\
=- \int_\Omega \alpha \left[ h_\gamma(Dv^+-w^+) -h_\gamma(Dv-w)- h'_\gamma(Dv-w)(D \delta_v - \delta_w)\right] (D \phi- \varphi)\\ 
- \int_\Omega \theta_1 \ \left[ h_\gamma(Dv^+-w^+)-h_\gamma(Dv-w) \right] (D \phi- \varphi) \,dx\\
- \int_\Omega \beta \left[h_\gamma(E w^+)-h_\gamma(E w)- h'_\gamma(E w) E \delta_w \right] :E \varphi \,dx\\ 
- \int_\Omega \theta_2 \ \left[ h_\gamma(E w_{\alpha+\theta})-h_\gamma(E w) \right]: E \varphi \,dx, \text{ for all } \Psi \in U.
\end{multline*}
Testing with $\Psi=\xi$ and using the monotonicity of $h_\gamma'(\cdot)$ we get that
\begin{multline*} 
\|\xi\|_U
\leq C \left\{ |\alpha| \left\| h_\gamma(Dv^+-w^+)-h_\gamma(Dv-w)- h'_\gamma(Dv-w)(D \delta_v - \delta_w) \right\|_{L^2} \right.\\ 
+|\theta_1| \left\| h_\gamma(Dv^+-w^+)-h_\gamma(Dv-w) \right\|_{L^2}\\
\left. +|\beta| \left\| h_\gamma(E w^+)-h_\gamma(E w)- h'_\gamma(E w) E \delta_w \right \|_{L^2} \right.\\ 
\left. +|\theta_2| \left\| h_\gamma(E w_{\alpha+\theta})-h_\gamma(E w) \right\|_{L^2} \right\},
\end{multline*}
for some generic constant $C >0$. Considering the differentiability and Lipschitz continuity of $h_\gamma'(\cdot)$, it then follows that 
\begin{multline} \label{eq: differentiability estimate}
\|\xi\|_U
\leq C \left( |\alpha|~ o (\left\| u^+-u \right\|_{1,p}) 
+|\theta_1| \left\| u_{\alpha+\theta}-u \right \|_{U} \right.\\
+\left. |\beta|~ o (\left\| w^+-w \right\|_{1,p}) 
+|\theta_2| \left\| w_{\alpha+\theta}-w \right \|_{\mathbb H^1(\Omega)} \right),
\end{multline}
where $\| \cdot \|_{1,p}$ stands for the norm in the space $\mathbb W^{1,p}(\Omega)$. From regularity results for second order systems (see \cite[Thm.~1,~Rem.~14]{Groeger1989}), it follows that
\begin{align*}
\left\| u^+-u \right\|_{1,p} & \leq L |\theta| \left( \|\mathrm{Div}~ h_\gamma(Dv -w)\|_{-1,p}+ \|h_\gamma(Dv -w)\|_{-1,p}+\|\mathrm{Div}~ h_\gamma(E w)\|_{-1,p} \right)\\
& \leq L |\theta| \left(2 \|h_\gamma(Dv -w)\|_{L^\infty}+\|h_\gamma(E w)\|_{L^\infty} \right)\\ & \leq \tilde L |\theta|,
\end{align*}
since $|h_\gamma (\cdot)| \leq 1$. Inserting the latter in estimate \eqref{eq: differentiability estimate}, we finally get that $$\| \xi \|_U= o(|\theta|).$$
\end{proof}

\begin{remark}
The Fr\'echet differentiability proof makes use of the quasilinear structure of the $\TGV^2$ variational form, making it difficult to extend to the ICTV model without further regularisation terms. For the latter, however, a Gateaux differentiability result may be obtained using the same proof technique as in \cite{de2013image}.
\end{remark}

\subsection{The adjoint equation}\label{sec:optsys}

Next, we use the Lagrangian formalism for deriving the adjoint equations for both the $\TGV^2$ and ICTV learning problems. Existence of a solution to the adjoint equation then follows from the well-posedness of the linearized equation.

Defining the Lagrangian associated to $\TGV^2$ learning problem by:
\begin{multline*}
\mathcal L(v,w,\alpha,\beta,p_1,p_2) = F(u) +\mu (v, p_1)_{H^1}+ \mu (w, p_2)_{\mathbb H^1}\\ + \int_\Omega \alpha h_\gamma (Dv - w)(D p_1- p_2)+ \int_\Omega \beta h_\gamma(E w) E p_2 + \int_\Omega (v-f) p_1,
\end{multline*} 
and taking the derivative with respect to the state variable $(v,w)$, we get the necessary optimality condition 
\begin{multline*}
\mathcal L_{(u,v)}'(u,v,\alpha,\beta,p_1,p_2)[(\delta_v, \delta_w)]= F'(u)\delta_u +\mu (p_1, \delta_v)_{H^1}+ \mu (p_2, \delta_w)_{\mathbb H^1}\\ + \int_\Omega \alpha h_\gamma' (Dv - w)(D \delta_v- \delta_w)(D p_1- p_2)\\+ \int_\Omega \beta h_\gamma' (E w) E \delta_w E p_2 + \int_\Omega p_1 \delta_v=0.
\end{multline*}
If $\delta_w=0$, then
\begin{equation*}
\mu (p_1, \delta_v)_{H^1}+ \int_\Omega \alpha h_\gamma' (Dv - w)(D p_1- p_2) D \delta_v+ \int_\Omega p_1 \delta_v=-\nabla_vF(u)\delta_v, 
\end{equation*}
whereas if $\delta_v=0$, then
\begin{multline*}
\mu (p_2, \delta_w)_{\mathbb H^1} - \int_\Omega \alpha h_\gamma' (Dv - w)(D p_1- p_2) \delta_w\\+ \int_\Omega \beta h_\gamma' (E w) \ E p_2 \ E \delta_w =-\nabla_w F(u)\delta_w.
\end{multline*}


Existence of a unique solution then follows from the transposition method, since the linearised equation is well-posed.

\begin{remark}
For the ICTV model it is possible to proceed formally with the Lagrangian approach. We recall that a necessary and sufficient optimality condition for the ICTV functional is given by
\begin{multline}
\mu (u, \phi)_{H^1}+ \mu (\nabla v, \nabla \varphi)_{\mathbb H^1} + \int_\Omega \alpha h_\gamma (Du - \nabla v)(D \phi- \nabla \varphi)\\+ \int_\Omega \beta h_\gamma(D \nabla v) D \nabla \varphi + \int_\Omega (u-f)\phi=0, \text{ for all }(\phi, \varphi) \in H^1(\Omega) \times \mathbb H^1(\Omega)
\end{multline}
and the correspondent Lagrangian functional $\mathcal L$ is given by
\begin{multline*}
\mathcal L(u,v,\alpha,\beta,p_1,p_2) = F(u) +\mu (u, p_1)_{H^1}+ \mu (\nabla v, \nabla p_2)_{\mathbb H^1}\\ + \int_\Omega \alpha h_\gamma (Du - \nabla v)(D p_1- \nabla p_2)+ \int_\Omega \beta h_\gamma(D \nabla v) D \nabla p_2 + \int_\Omega (u-f) p_1.
\end{multline*}
Deriving the Lagrangian with respect to the state variable $(u,v)$ and setting it equal to zero yields
\begin{multline*}
\mathcal L_{(u,v)}'(u,v,\alpha,\beta,p_1,p_2)[(\delta_u, \delta_v)]= F'(u)\delta_u +\mu (p_1, \delta_u)_{H^1}+ \mu (\nabla p_2, \nabla \delta_v)_{\mathbb H^1}\\ + \int_\Omega \alpha h_\gamma' (Du - \nabla v)(D \delta_u- \nabla \delta_v)(D p_1- \nabla p_2)\\+ \int_\Omega \beta h_\gamma' (D \nabla v) D \nabla \delta_v D \nabla p_2 + \int_\Omega p_1 \delta_u=0.
\end{multline*}
By taking succesively $\delta_v=0$ and $\delta_u=0$, the following system is obtained
\begin{subequations}
\begin{multline}
\mu (p_1, \delta_u)_{H^1}+ \int_\Omega \alpha h_\gamma' (Du - \nabla v)(D p_1- \nabla p_2) D \delta_u+ \int_\Omega p_1 \delta_u=-F'(u)\delta_u.
\end{multline}
\begin{multline}
\mu (\nabla p_2, \nabla \delta_v)_{\mathbb H^1} + \int_\Omega \alpha h_\gamma' (Du - \nabla v)(D p_1- \nabla p_2) \nabla \delta_v\\+ \int_\Omega \beta h_\gamma' (D \nabla v) D \nabla p_2 D \nabla \delta_v =0.
\end{multline}
\end{subequations}
\end{remark}

\subsection{Optimality condition}
Using the differentiability of the solution operator and the well-posedness of the adjoint equation, we derive next an optimality system for the characterization of local minima of the bilevel learning problem. Besides the optimality condition itself, a gradient formula arises as byproduct, which is of importance in the design of solution algorithms for the learning problems.

\begin{theorem} \label{thm: optimality system}
Let $(\bar \alpha, \bar \beta) \in \mathbb R^2_+$ be a local optimal solution for problem \eqref{eq:learn}. Then there exist Lagrange multipliers $\Pi \in U$ and $\lambda_1, \lambda_2 \in L^2(\Omega)$ such that the following system holds:
\begin{subequations}
\begin{multline}
a(u, \Psi)+\alpha \int_\Omega h_\gamma(Dv-w)(D \phi- \varphi) \,dx\\ + \beta \int_\Omega h_\gamma(E w) E \varphi \,dx +2 \int_\Omega(v-f)\phi \,dx=0, \text{ for all } \Psi \in H^1(\Omega) \times \mathbb H^1(\Omega),
\end{multline}
\begin{multline} \label{eq: adjoint equation TGV}
a(\Pi, \Psi)+\alpha \int_\Omega h_\gamma' (Dv-w)(D p_1-p_2)(D \phi- \varphi) \,dx\\ + \beta \int_\Omega h_\gamma' (E w) \ E p_2 \ E \varphi \,dx +2 \int_\Omega p_1 \phi \,dx=-F_u(u)[\Psi], \text{ for all } \Psi \in H^1(\Omega) \times \mathbb H^1(\Omega),
\end{multline}
\begin{equation}
\lambda_1= \int_\Omega h_\gamma (Dv-w)(D p_1 -p_2)
\end{equation}
\begin{equation}
\lambda_2= \int_\Omega h_\gamma (Ew) \ E p_2
\end{equation}
\begin{equation}
\lambda_1 \geq 0, \qquad \lambda_2 \geq 0
\end{equation}
\begin{equation}
\lambda_1 \cdot \bar \alpha = \lambda_2 \cdot \bar \beta=0.
\end{equation}
\end{subequations}
\end{theorem}
\begin{proof}
Consider the reduced cost functional $\mathcal F(\alpha, \beta)=F(u(\alpha, \beta)).$  The bilevel optimization problem can then be formulated as
\begin{align*}
& \min_{(\alpha, \beta) \in C}  \mathcal F(\alpha, \beta),
\end{align*}
where $\mathcal F: \mathbb R^{2} \to \mathbb R$ and $C$ corresponds to the positive orthant in $\mathbb R^2$. From \cite[Thm.~3.1]{ZoweK1979}, there exist multipliers $\lambda_1, \lambda_2$ such that
\begin{align*}
\lambda_1= \nabla_\alpha \mathcal F(\bar \alpha, \bar \beta)\\
\lambda_2= \nabla_\beta \mathcal F(\bar \alpha, \bar \beta)\\
\lambda_1 \geq 0, \quad \lambda_2 \geq 0\\
\lambda_1 \cdot \bar \alpha = \lambda_2 \cdot \bar \beta=0,
\end{align*}

By taking the derivative with respect to $(\alpha, \beta)$ and denoting by $u'$ the solution to the linearized equation \eqref{eq: linearized equation} we get, together with the adjoint equation \eqref{eq: adjoint equation TGV}, that 
\begin{align*}
\mathcal F'(\alpha, \beta)[\phi]& =F_u(u)u'(\alpha,\beta)[\phi]\\
&= -a(\Pi,u')- \alpha \int_\Omega h_\gamma'(Dv-w)(D p_1 -p_2 )(D v'-w')\\& \qquad -\beta \int_\Omega h_\gamma'(E w) E p_2 \ E w'-2 \int_\Omega p_1 v'\\
&= -a(u',\Pi) - \alpha \int_\Omega h_\gamma'(Dv-w)(D v'-w') (D p_1 -p_2 )\\& \qquad -\beta \int_\Omega h_\gamma'(E w) E w' \ E p_2-2 \int_\Omega v' p_1
\end{align*}
which, taking into account the linearized equation, yields
\begin{equation} \label{eq: derivative of reduced cost}
\mathcal F'(\alpha, \beta)[\phi]=\phi_1 \int_\Omega h_\gamma(Dv-w)(D p_1 -p_2 )+ \phi_2 \int_\Omega h_\gamma(E w)E p_2.
\end{equation}
Altogether we proved the result.
\end{proof}

\begin{remark}
From the existence result (see Remark \ref{rem: existence}), we actually know that, under some assumptions, $\bar \alpha$ and $\bar \beta$ are strictly greater than zero. This implies that the multipliers $\lambda_1=\lambda_2=0$ and the problem is of unconstrained nature. This plays an important role in the design of solution algorithms, since only a mild treatment of the constraints has to be taken into account, as will be showed in Section 6.
\end{remark}

\section{Numerical algorithms}
\label{sec:denoising}
In this section we propose a second order quasi-Newton method for the solution of the learning problem with scalar regularisation parameters. The algorithm is based on a BFGS update, preserving the positivity of the iterates through the line search strategy and updating the matrix cyclically depending on the satisfaction of the curvature condition. For the solution of the lower level problem, a semismooth Newton method with a properly modified Jacobi matrix is considered. Moreover, warm initialisation strategies have to be taken into account in order to get convergence for the $\TGV^2$ problem.
The developed algorithm is also extended to a simple linear polynomial case. 


\subsection{BFGS algorithm}\label{sec:bfgs}
Thanks to the gradient characterization obtained in Theorem \ref{thm: optimality system}, we next devise a BFGS algorithm to solve the bilevel learning problems. We employ a few technical tricks to ensure convergence of the classical method. In particular,
for numerical stability we need to avoid the boundary of the constraint set on the parameters, 
so we pick $0 < \theta < \Theta$, considered numerically almost zero or infinity, respectively,
and require the box constraints
\begin{equation}
    \label{eq:bfgs-constr}
    \theta \le \alpha, \beta \le \Theta.
\end{equation}
We also limit the step length to get at most a fraction closer to the boundary. 
As we show in \cite{tuomov-interior} the solution is in the interior for the regularisation 
and cost functionals we are interested in. Below this limit, we use Armijo line search. 

Moreover, the good behaviour of the BFGS method depends upon the
BFGS matrix staying positive definite. This would be ensured by the Wolfe conditions, but because 
of our step length limitation,
the curvature condition is not necessarily satisfied. (The Wolfe conditions are guaranteed
to be satisfied for some step length $\sigma$, if our domain is unbounded, but the range 
where the step satisfies the criterion, may be beyond our maximum step length, and
is not necessarily satisfied closer to the current point.)
Instead we skip the BFGS update if the curvature is negative. 

Overall our learning algorithm may be written as follows.

\begin{algorithm}[BFGS for denoising parameter learning]
    \label{algorithm:bfgs-learn}
    Pick Armijo line search constant $c$, and target residual $\rho$. 
    Pick initial iterate $(\alpha^0,\beta^0)$.
    Solve the denoising problem \eqref{eq:denoise-huber1}
    for $(\alpha,\beta)=(\alpha^0,\beta^0)$, yielding $u^0$.
    Initialise $B^1=I$.
    Set $i \defeq 0$, and  iterate the following steps:
    \begin{enumerate}
        \item \label{step:adjoint}
            Solve the adjoint equation \eqref{eq: adjoint equation TGV} for $\Pi^i$,
            and calculate $\grad \mathcal F (\alpha^i,\beta^i)$ from \eqref{eq: derivative of reduced cost}.
        \item If $i \ge 2$ do the following:
            \begin{enumerate}
                \item Set $s \defeq (\alpha^i,\beta^i)-(\alpha^{i-1}, \beta^{i-1})$,
                    and $r \defeq \grad \mathcal F(\alpha^i,\beta^i)-\grad \mathcal F(\alpha^{i-1},\beta^{i-1})$.
                \item
                    Perform the BFGS update
                    \[
                        B^i \defeq \begin{cases}
                            B^{i-1}, 
                            & s^Tr < 0, \\
                            B^{i-1} - \frac{B^{i-1} s \otimes B^{i-1}s}{t^T B^{i-1} s} + \frac{r \otimes r}{s^Tr}
                            & s^Tr \ge 0.
                            \end{cases}
                    \]
            \end{enumerate}
        \item Compute $\delta_{\alpha, \beta}$ from
            \[
                B^i \delta_{\alpha, \beta} = g^i.
            \]
        \item Initialise
            $\sigma \defeq \min\{1, \sigma_{\max}/2\}$, where
            \[
                \sigma_{\max} \defeq \max \{ \sigma \ge 0 \mid (\alpha^i, \beta^i)+\sigma \delta_{\alpha, \beta} \text{ satisfies } \eqref{eq:bfgs-constr}\}.
            \]
            Repeat the following:
            \begin{enumerate}
                \item 
                    \label{item:linesearch-1}
                    Let $(\alpha_\sigma, \beta_\sigma) \defeq (\alpha^i, \beta^i)+\sigma \delta_{\alpha, \beta}$, 
                    and solve the denoising problem \eqref{eq:denoise-huber1}
                    for $(\alpha, \beta)=(\alpha_\sigma, \beta_\sigma)$, yielding $u_\sigma$.
                \item If the residual $\norm{(\alpha_\sigma, \beta_\sigma) - (\alpha^i, \beta^i)}/\norm{(\alpha_\sigma, \beta_\sigma)} < \rho$ do the following:
                    \begin{enumerate}
                        \item If $\min_\sigma \mathcal F(\alpha_\sigma, \beta_\sigma) < \mathcal F(\alpha^i, \beta^i)$
                        over all $\sigma$ tried, choose $\sigma^*$ the minimiser, 
                        set $(\alpha^{i+1}, \beta^{i+1}) \defeq (\alpha_{\sigma^*}, \beta_{\sigma^*})$, $u^{i+1} \defeq u_{\sigma^*}$, and
                        continue from Step \ref{step:res-check}
                        \item Otherwise end the algorithm with solution $(\alpha^*, \beta^*) \defeq (\alpha^i, \beta^i)$.
                    \end{enumerate}
                \item Otherwise, if Armijo condition $\mathcal F(\alpha_\sigma, \beta_\sigma) \le \mathcal F(\alpha^i, \beta^i) + \sigma c \grad \mathcal F(\alpha^i,\beta^i)^T \delta_{\alpha, \beta}$
                    holds, set $(\alpha^{i+1}, \beta^{i+1}) \defeq (\alpha_{\sigma}, \beta_{\sigma})$,
                    $u^{i+1} \defeq u_{\sigma}$, and continue from Step \ref{step:res-check}. 
                \item In all other cases, set $\sigma \defeq \sigma/2$ 
                    and continue from Step \ref{item:linesearch-1}.
            \end{enumerate}
        \item \label{step:res-check}
            If the residual $\norm{(\alpha^{i+1}, \beta^{i+1}) - (\alpha^i, \beta^i)}/\norm{(\alpha^{i+1}, \beta^{i+1})} < \rho$, end
            the algorithm with $(\alpha^* , \beta^*) \defeq (\alpha^{i+1}, \beta^{i+1})$.
            Otherwise continue from Step \ref{step:adjoint} with $i \defeq i+1$.
    \end{enumerate}
\end{algorithm}
Step (4) ensures that the iterates remain feasible, without making use of a projection step. This is justified since it's been analytically proved that the optimal parameters are greater than zero (see \cite{tuomov-interior}).

\subsection{An infeasible semi-smooth Newton method}\label{sec:seminewton}
In variational form, the $\TGV^2$ denoising problem can be written as
\begin{equation*}
\mu \int_\Omega (Dv \cdot D \phi+ v \phi)+ \int_\Omega \alpha h_\gamma (Dv - w) D \phi + \int_\Omega (v-f) \phi =0, \quad \forall \phi \in H^1(\Omega)
\end{equation*}
\begin{multline*}
\mu \int_\Omega (Ew : E \varphi+ w \varphi) - \int_\Omega \alpha h_\gamma (Dv - w) D \varphi \\+ \int_\Omega \beta h_\gamma (E w) \ E \varphi =0, \quad \forall \varphi \in \mathbb H^1(\Omega)
\end{multline*}
or, in general abstract primal-dual form, as 
\begin{subequations}
\label{eq:ssn-sys}
\begin{gather}
    \label{eq:ssn-sys1}
    L u +  \sum_{i=1}^{\NA} A_j^* q_j = f \quad \text{ in } \Omega
    \\
    \label{eq:ssn-sys2}
    \max\{1/\gamma, \abs{[A_j u](x)}_2\} q_j(x) - \alpha_j [A_j u](x) = 0
    \text{ a.e. in }\Omega,
    \quad
    j=1,\ldots,\NA.
\end{gather}
\end{subequations}
where $L \in \mathcal L (\SPACEuHdual,\SPACEuHdual')$ is a second order linear elliptic operator, $A_j, ~j=1, \dots, N$, are linear operators acting on $u$ and $q_j(x), ~j=1, \dots, N$, correspond to the dual multipliers. 

Let us set
\[
    \maxO_j(u) \defeq \max\{1/\gamma, \abs{[A_j u](x)}_2\}.
\]
Let us also define the diagonal application 
$\DiagO(u): \SPACEuLtwo \to \SPACEuLtwo$ by
\[
    [\DiagO(u) q](x) = u(x) q(x),
    \quad
    (x \in \Omega)
\]
We may derive $\grad_u [\DiagO(\maxO_j(u)) q_j]$ being defined by
\[
    \grad_u [\DiagO(\maxO_j(u)) p_j]
    =
    A_j^* \DiagO(q_j) \nO(A_j u)
    \quad
    \text{where}
    \quad
    \nO(z)
    \defeq
    \begin{cases}
        0, & \abs{z(x)}_2 < 1/\gamma \\
        \frac{z(x)}
             {\abs{z(x)}_2}, & \abs{z(x)}_2 \geq 1/\gamma. \\
    \end{cases}
\]
Then \eqref{eq:ssn-sys1}, \eqref{eq:ssn-sys2} may be written as
\begin{gather}
    \notag
    L u + \sum_{i=1}^{\NA} A_j^* q_j = f \quad \text{ in } \Omega
    \\
    \notag
    \DiagO(\maxO_j(u)) q_j - \alpha_j A_j u = 0,
    \quad
    \text{ a.e. in }
    \Omega,
    \quad
    (j=1,\ldots,\NA).
\end{gather}
Linearising, we obtain the system
\begin{equation}
    \label{eq:ssn}
    \tag{SSN-1}
    \begin{pmatrix}
        L & A_1^* & \ldots & A_\NA^* \\
        - \alpha_1 A_1 + \nO(A_1 u) \DiagO(q_1) A_1 & \DiagO(\maxO_j(u)) & 0 & 0 \\
        \vdots & 0 & \ddots & 0 \\
        - \alpha_\NA A_\NA + \nO(A_\NA u) \DiagO(q_\NA) A_\NA & 0 & 0 & \DiagO(\maxO_\NA(u)) \\
    \end{pmatrix}
    \begin{pmatrix}
        \Step u \\
        \Step q_1 \\
        \vdots \\
        \Step q_\NA \\
    \end{pmatrix}
    =
    R
\end{equation}
where
\[
    R \defeq
    \begin{pmatrix}
        -L u - \sum_{i=1}^{\NA} A_j^* q_j + f \\
        \alpha_1 A_1 u - \DiagO(\maxO_1(u)) q_1\\
        \vdots \\
        \alpha_\NA A_\NA u - \DiagO(\maxO_\NA(u)) q_\NA
    \end{pmatrix}.
\]
The semi-smooth Newton method solves \eqref{eq:ssn} at a current iterate $(u^i, q_1^i, \ldots q_\NA^i)$.
It then updates
\begin{equation}
    \tag{SSN-2}
    \label{eq:ssn-step}
    (u^{i+1}, \tilde q_1^{i+1}, \ldots \tilde q^{i+1}_\NA) := (u^i + \tau \Step u, q_1^i + \tau \Step q_1, q_\NA^i + \tau \Step q_\NA), 
\end{equation}
for a suitable step length $\tau$, allowing $\tilde q^{i+1}$ to become infeasible in the
process. That is, it may hold that $\abs{\tilde q_j^{i+1}(x)}_2 > \alpha_j$, which may lead to non-descent directions. In order to globalize the method, one projects
\begin{equation}
    \label{eq:ssn-proj}
    \tag{SSN-3}
    q_j^{i+1} \defeq \projO(\tilde q_j^{i+1}; \alpha_j),
    \quad
    \text{where}
    \quad
    \projO(q, \alpha)(x) \defeq \sign(q(x)) \min\{\alpha, \abs{q(x)}\},
\end{equation}
in the building of the Jacobi matrix. Following \cite{hintermuller2006infeasible,sun1997newton}, it can be shown that a discrete version of
the method \eqref{eq:ssn}--\eqref{eq:ssn-proj} converges globally and locally 
superlinearly near a point where the subdifferentials of the
operator on $(u, q_1, \ldots q_\NA)$ corresponding \eqref{eq:ssn-sys}
are non-singular. Further dampening as in \cite{hintermuller2006infeasible}
guarantees local superlinear convergence at any point.
We do not represent the proof, as going into the discretisation and
dampening details would expand this work considerably.

\begin{remark}
    \label{rem:ssn-simplify}
    The system \eqref{eq:ssn} can be simplified, which is crucial to obtain
    acceptable performance with $\TGV^2$.
    Indeed observe that $B$ is invertible, so we may solve $\Step u$
    from
    \begin{equation}
        \label{eq:delta-u-upd}
        B \Step u  = R_1 - \sum_{j=1}^N A_j^* \Step q_j.
    \end{equation}
    Thus we may simplify $\Step u$ out of \eqref{eq:ssn}, 
    and only solve for $\Step q_1, \ldots, \Step q_N$
    using a reduced system matrix.
    Finally we calculate $\Step u$ from \eqref{eq:delta-u-upd}.
\end{remark}


For the denoising sub-problem \eqref{eq:denoise-huber1} we use the method
\eqref{eq:ssn}--\eqref{eq:ssn-proj} with the reduced system matrix
of Remark \ref{rem:ssn-simplify}. Here, we denote by $z$ in the case of TGV$^2$ the parameters 
$$
z=(v,w),
$$ 
and in the case of ICTV 
$$
z=(u,v).
$$ 
For the calculation of the step length
$\tau$, we use Armijo line search with parameter $c=1\ee^{-4}$.
We end the SSN iterations when
\[
    \tau\frac{\norm{\Step y^i}}{\max\{1,\norm{y^i}\}} \le 1\ee^{-5},
\]
where $\Step y^i=(\Step z^i,\Step q_1^i, \ldots, \Step q_N^i)$,
and $y^i=(z^i, q_1^i, \ldots,  q_N^i)$.

\newcommand{\resplotc}[4][]{
    \begin{subfigure}[t]{0.45\textwidth}%
    \resplotxx[#1]{resimg/#2}
    \input{resimg/#2-vals.tex}%
    \caption{\methodnamex{#4}}
    \ifdefined\subfigprefix\label{\subfigprefix:#2}\else\relax\fi 
    \end{subfigure}
    }

\newcommand{\restablc}[5][]{%
   \input{resimg/#2-vals.tex}%
   \methodnamex{#4} &  
   \mathname{#3} & 
   \RESssim & 
   \RESpsnr &
   \RESiters &
   $\REStime$s &
   \def\arg{#5}%
   \def\nofig{}%
   \ifx\arg\nofig\else\ref{#5}(\subref{#5:#2})\fi \\%
}

\subsection{Warm initialisation}\label{sec:numdisc}

In our numerical experimentation we generally found 
Algorithm \ref{algorithm:bfgs-learn} to perform well for learning the regularisation
parameter for $\TV$ denoising as was done in \cite{de2013image}. For learning the two (or even more) regularisation parameters for
$\TGV^2$ denoising, we found that a warm initialisation is needed to obtain convergence. More specifically, we use $\TV$
as an aid for discovering both the initial iterate $(\alpha^0,\beta^0)$ as well as the initial
BFGS matrix $B^1$. This is outlined in the following algorithm.

\begin{algorithm}[BFGS initialisation for $\TGV^2$ parameter learning]
    \label{algorithm:bfgs-learn-tgv2}
    Pick a heuristic factor $\delta_0 > 0$. Then do the following:
    \begin{enumerate}
        \item Solve the corresponding problem for $\TV$ using Algorithm \ref{algorithm:bfgs-learn}.
            This yields optimal $\TV$ denoising parameter $\alpha_\TV^*$, as well as the 
            BFGS estimate $B_\TV$ for $\grad^2 \mathcal F (\alpha_\TV^*)$.
        \item Run Algorithm \ref{algorithm:bfgs-learn} for $\TGV^2$ with initialisation
            $(\alpha^0,\beta^0) \defeq (\alpha_\TV^* \delta_0, \alpha_\TV^*)$,
            and initial BFGS matrix $B^1 \defeq \mathrm{diag}(B_\TV \delta_0, B_\TV)$.
    \end{enumerate}
\end{algorithm}

With $\Omega=(0, 1)^2$, we pick $\delta_0=1/\pixels$, where the original discrete image has $\pixels \times \pixels$ pixels. This corresponds to the heuristic \cite{tuomov-dtireg,tuomov-phaserec} that if $\pixels \approx 128$ or $256$ and the discrete image is mapped into 
the corresponding domain $\Omega=(0, \pixels)^2$ directly (corresponding to spatial step size of one in  the discrete gradient operator), then  $\beta \in (\alpha, 1.5 \alpha)$ tends to be a good choice. We will later verify this through the use of our algorithms. 
Now, if $f \in \BVspace((0, \pixels)^2)$ is rescaled to $\BVspace((0, 1)^2)$, i.e. $\tilde f(x) \defeq f(x/\pixels)$, then with $\tilde u(x) \defeq u(x/\pixels)$ and $\tilde w(x) \defeq w(x/\pixels)/\pixels$, we have
\begin{multline}
    \frac{1}{2}\norm{f-u}_{L^2((0, \pixels)^2)}^2
    +\alpha\norm{Du-w}_{\Meas((0, \pixels)^2; \R^2)}
    +\beta\norm{Ew}_{\Meas((0, \pixels)^2; \R^{2 \times 2})}
    \\
    =
    n^2\left(
        \frac{1}{2}\norm{\tilde f-\tilde u}_{L^2((0,1)^2)}^2
        +n\alpha\norm{D\tilde u-\tilde w}_{\Meas((0,1)^2; \R^2)}
        +n^2\beta\norm{E\tilde w}_{\Meas((0,1)^2; \R^{2 \times 2})}
    \right).
    \end{multline}
This introduces the factor $1/\pixels=\abs{\Omega}^{-1/2}$ between rescaled $\alpha$, $\beta$.


\section{Experiments} \label{sec:experiments}
In this section we present some numerical experiments to verify the theoretical properties of the bilevel learning problems and the efficiency of the proposed solution algorithms. In particular, we exhaustively compare the performance of the new proposed cost functional with respect to well-known quality measures, showing a better behaviour of the new cost for the chosen tested images. The performance of the proposed BFGS algorithm, combined with the semismooth Newton method for the lower level problem, is also examined. 

\subsection{Gaussian denoising}
\label{sec:denoising}
We tested Algorithm \ref{algorithm:bfgs-learn} for $\TV$ 
and Algorithm \ref{algorithm:bfgs-learn-tgv2} for $\TGV^2$
Gaussian denoising parameter learning on various images. Here we report the
results for two images, the parrot image in Figure \ref{fig:res-dataset2:kodim23gray-crop}, and the geometric image in Figure \ref{fig:res-dataset11}.
We applied synthetic noise to the original images, such that the PSNR of the parrot image is $24.7$, and the PSNR of the geometric image is $24.8$.

In order to learn the regularisation parameter $\alpha$ for $\TV$, we 
picked initial $\alpha^0=0.1/\pixels$. For $\TGV^2$ initialisation
by $\TV$ was used as in Algorithm \ref{algorithm:bfgs-learn}.
We chose the other parameters of Algorithm \ref{algorithm:bfgs-learn}
as $c=1\ee^{-4}$, $\rho=1\ee^{-5}$, $\theta=1\ee{-8}$, and $\Theta=10$.
For the SSN denoising method the parameters $\gamma=100$ and
$\mu=1\ee^{-10}$ were chosen. 

We have included results for both the
$L^2$-squared cost functional $\costltwo$ and the Huberised
total variation cost functional $\costhubertv$. The learning
results are reported in Table \ref{table:res-dataset2}
for the parrot images, and Table \ref{table:res-dataset11}
for the geometric image. The denoising results with the discovered
parameters can be found in the aforementioned 
Figure \ref{fig:res-dataset2} and Figure \ref{fig:res-dataset11}.
We report the resulting optimal parameter values, the cost functional
value, PSNR, SSIM \cite{wang2004ssim}, as well as the number of
iterations taken by the outer BFGS method.

Our first observation is that all approaches successfully
learn a denoising parameter that gives a good-quality denoised
image. Secondly, we observe that the gradient cost functional
$\costhubertv$ performs visually and in terms of SSIM significantly
better for $\TGV^2$  parameter learning than the cost functional
$\costltwo$. In terms of PSNR the roles are reversed, as should be,
since the $\costltwo$ is equivalent to PSNR. This again confirms
that PSNR is a poor quality measure for images. For $\TV$ there
is no significant difference between different cost functionals
in terms of visual quality, although the PSNR and SSIM differ.

We also observe that the optimal $\TGV^2$ parameters 
$(\alpha^*, \beta^*)$ generally satisfy
$\beta^*/\alpha^* \in (0.75, 1.5)/\pixels$.
This confirms the earlier observed heuristic that
if $\pixels \approx 128,\, 256$ then
$\beta \in (1, 1.5) \alpha$ tends to be a good choice. 
As we can observe from Figure \ref{fig:res-dataset2} and
Figure \ref{fig:res-dataset11}, this optimal $\TGV^2$ parameter
choice also avoids the stair-casing effect that can be observed
with $\TV$ in the results. 

\begin{figure}
    \begin{subfigure}[t]{0.47\textwidth}
        \includegraphics[width=\textwidth]{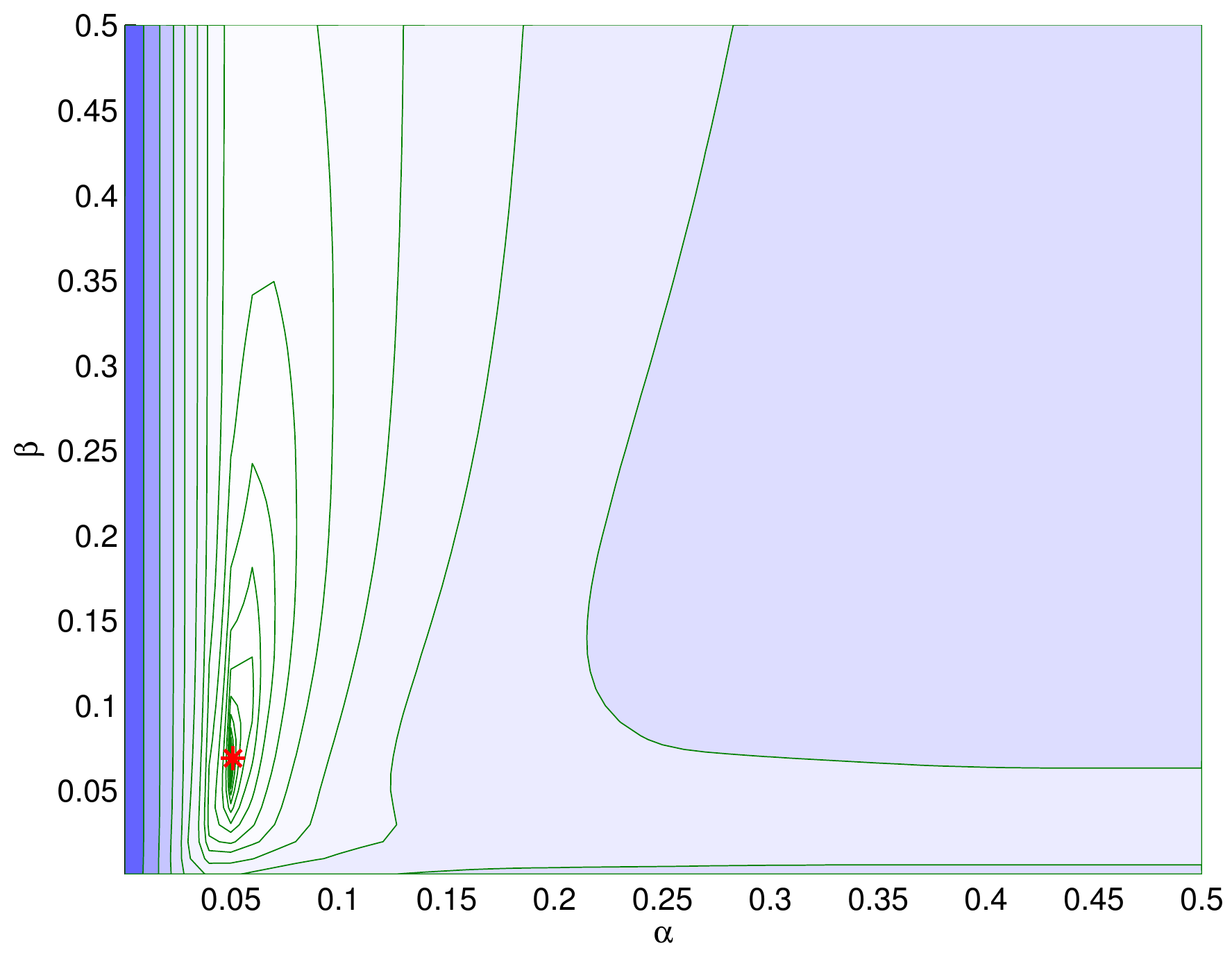}
        \caption{Parrot, $\TGV^2$, {\costhubertv} cost functional}
    \end{subfigure}
    \hfill
    \begin{subfigure}[t]{0.47\textwidth}
        \includegraphics[width=\textwidth]{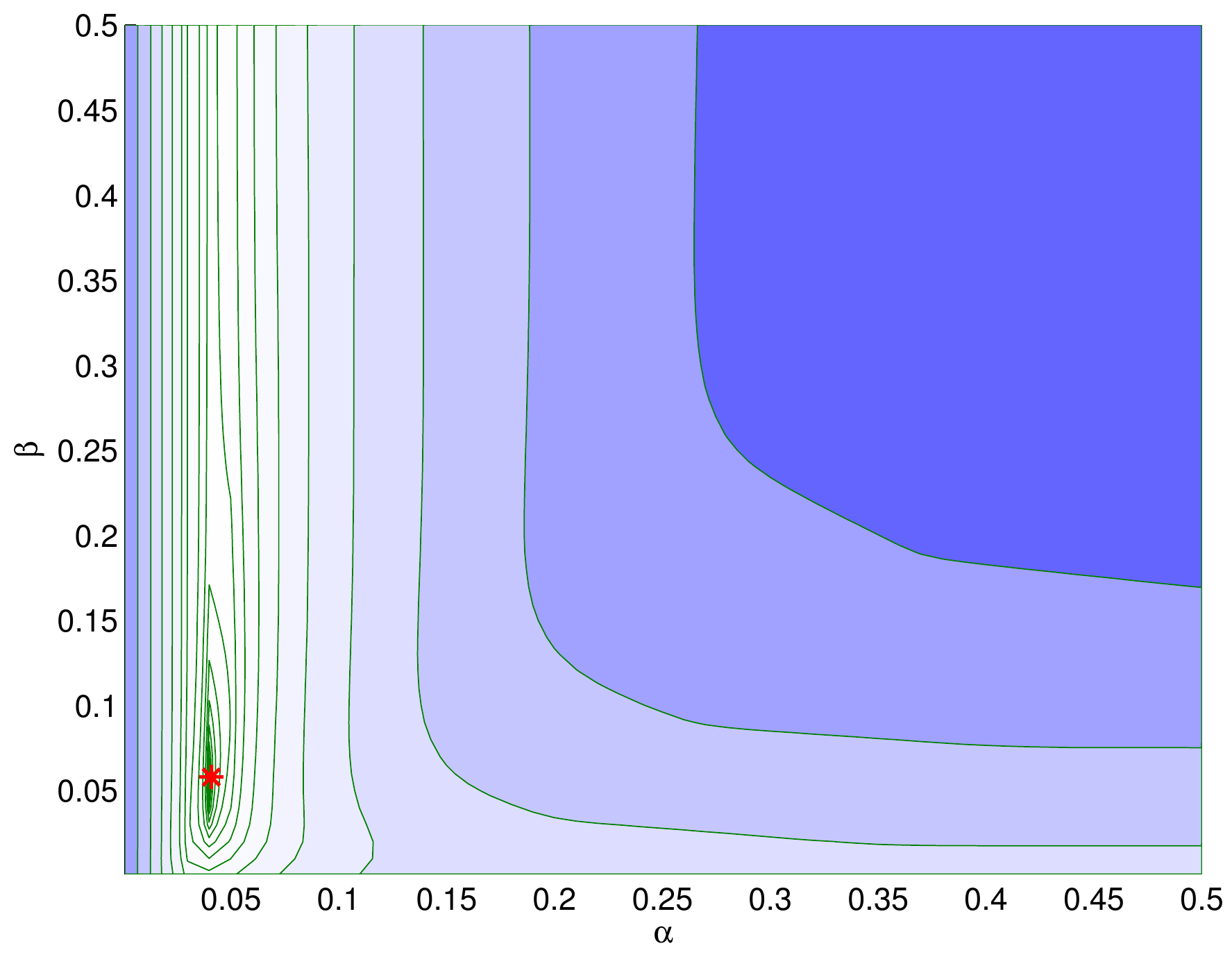}
        \caption{Parrot, $\TGV^2$, {\costltwo} cost functional}
    \end{subfigure}
    \caption{Cost functional value versus $(\alpha, \beta)$ for $\TGV^2$ denoising, 
        for the parrot test images, 
        for both {\costltwo} and {\costhubertv}
        cost functionals.
        The illustrations are contour plots of function value versus $(\alpha, \beta)$.
        }
    \label{fig:landscape-tgv2}
\end{figure}
In Figure \ref{fig:landscape-tgv2}, we have plotted by the red star the
discovered regularisation parameter $(\alpha^*, \beta^*)$ reported in
Figure \ref{fig:res-dataset2}. Studying the location of the red star, we
may conclude that Algorithm \ref{algorithm:bfgs-learn} 
and Algorithm \ref{algorithm:bfgs-learn-tgv2} manage
to find a nearly optimal parameter in very few BFGS 
iterations.

\begin{figure}
    \setlength{\imw}{0.33\textwidth}
    \centering
    \def\subfigprefix{fig:res-dataset2}
    \inplot{kodim23gray-crop}
            {Original image}
    \inplot{kodim23gray-crop-noisy}
            {Noisy image}
    \\
    \resplot{tgv2-huberonly-dataset2-linesearch1-bfgs1-beta0-uinitreg_0.100000_}
            {tgv2}{huberonly}{$\TVINIT$}
    \resplot{tgv2-l2-dataset2-linesearch1-bfgs1-beta0-uinitreg_0.100000_}
            {tgv2}{l2}{$\TVINIT$}
    \resplot{ictv-huberonly-dataset2-linesearch1-bfgs1-beta0-uinitreg_0.100000_}
            {ictv}{huberonly}{$\TVINIT$}
    \resplot{ictv-l2-dataset2-linesearch1-bfgs1-beta0-uinitreg_0.100000_}
            {ictv}{l2}{$\TVINIT$}
    \resplot{tv-huberonly-dataset2-linesearch1-bfgs1-beta0-uinitreg_0.100000_}
            {tv}{huberonly}{$0.1/\pixels$}
    \resplot{tv-l2-dataset2-linesearch1-bfgs1-beta0-uinitreg_0.100000_}
            {tv}{l2}{$0.1/\pixels$}
    \caption{Optimal denoising results for initial guess $\alphavec=\TVINIT$ 
             for $\TGV^2$ and $\alphavec=0.1/\pixels$ for $\TV$}
    \label{\subfigprefix}
\end{figure}

\begin{table}
    \caption{Quantified results for the parrot image ($\pixels=256=\text{image width/height in pixels}$)}   
    \label{table:res-dataset2}
    \centering
    \begin{tabular}{lll|lllll|l}
    Denoise & 
    Cost & 
    Initial $(\alpha, \beta)$ & 
    Result $(\alpha^*, \beta^*)$ & 
    Cost & 
    SSIM & 
    PSNR &
    Its. &
    Fig. \\
    \hline
    \restabll{tgv2-huberonly-dataset2-linesearch1-bfgs1-beta0-uinitreg_0.100000_}
            {tgv2}{huberonly}{$\TVINIT$}{fig:res-dataset2}
    \restabll{tgv2-l2-dataset2-linesearch1-bfgs1-beta0-uinitreg_0.100000_}
            {tgv2}{l2}{$\TVINIT$}{fig:res-dataset2}
    \restabll{ictv-huberonly-dataset2-linesearch1-bfgs1-beta0-uinitreg_0.100000_}
            {ictv}{huberonly}{$\TVINIT$}{fig:res-dataset2}
    \restabll{ictv-l2-dataset2-linesearch1-bfgs1-beta0-uinitreg_0.100000_}
            {ictv}{l2}{$\TVINIT$}{fig:res-dataset2}
    \restabll{tv-huberonly-dataset2-linesearch1-bfgs1-beta0-uinitreg_0.100000_}
            {tv}{huberonly}{$0.1/\pixels$}{fig:res-dataset2}
    \restabll{tv-l2-dataset2-linesearch1-bfgs1-beta0-uinitreg_0.100000_}
            {tv}{l2}{$0.1/\pixels$}{fig:res-dataset2}
    \end{tabular}
\end{table}

\begin{figure}
    \setlength{\imw}{0.33\textwidth}
    \centering
    \def\subfigprefix{fig:res-dataset11}
    \inplot
            {synth}
            {Original image}
    \inplot{synth-noisy}
            {Noisy image}
    \\
    \resplot{tgv2-huberonly-dataset11-linesearch1-bfgs1-beta0-uinitreg_0.200000_}
            {tgv2}{huberonly}{$\TVINIT$}
    \resplot{tgv2-l2-dataset11-linesearch1-bfgs1-beta0-uinitreg_0.200000_}
            {tgv2}{l2}{$\TVINIT$}
    \resplot{ictv-huberonly-dataset11-linesearch1-bfgs1-beta0-uinitreg_0.200000_}
            {ictv}{huberonly}{$\TVINIT$}
    \resplot{ictv-l2-dataset11-linesearch1-bfgs1-beta0-uinitreg_0.200000_}
            {ictv}{l2}{$\TVINIT$}
    \resplot{tv-huberonly-dataset11-linesearch1-bfgs1-beta0-uinitreg_0.200000_}
            {tv}{huberonly}{$0.1/\pixels$}
    \resplot{tv-l2-dataset11-linesearch1-bfgs1-beta0-uinitreg_0.200000_}
            {tv}{l2}{$0.1/\pixels$}
    \caption{Optimal denoising results for initial guess $\alphavec=\TVINIT$ 
             for $\TGV^2$ and $\alphavec=0.2/\pixels$ for $\TV$}
    \label{\subfigprefix}
\end{figure}

\begin{table}
    \caption{Quantified results for the synthetic image ($\pixels=256=\text{image width/height in pixels}$)}
    
    \label{table:res-dataset11}
    \centering
    \begin{tabular}{lll|lllll|l}
    Denoise & 
    Cost & 
    Initial $\alphavec$ & 
    Result $\alphavec^*$ & 
    Value & 
    SSIM & 
    PSNR &
    Its. &
    Fig. \\
    \hline
    \restabll{tgv2-huberonly-dataset11-linesearch1-bfgs1-beta0-uinitreg_0.200000_}
            {tgv2}{huberonly}{$\TVINIT$}{fig:res-dataset11}
    \restabll{tgv2-l2-dataset11-linesearch1-bfgs1-beta0-uinitreg_0.200000_}
            {tgv2}{l2}{$\TVINIT$}{fig:res-dataset11}
    \restabll{ictv-huberonly-dataset11-linesearch1-bfgs1-beta0-uinitreg_0.200000_}
            {ictv}{huberonly}{$\TVINIT$}{fig:res-dataset11}
    \restabll{ictv-l2-dataset11-linesearch1-bfgs1-beta0-uinitreg_0.200000_}
            {ictv}{l2}{$\TVINIT$}{fig:res-dataset11}
    \restabll{tv-huberonly-dataset11-linesearch1-bfgs1-beta0-uinitreg_0.200000_}
            {tv}{huberonly}{$0.1/\pixels$}{fig:res-dataset11}
    \restabll{tv-l2-dataset11-linesearch1-bfgs1-beta0-uinitreg_0.200000_}
            {tv}{l2}{$0.1/\pixels$}{fig:res-dataset11}
    \end{tabular}
\end{table}

\subsection{Statistical testing}

To obtain a statistically significant outlook to the performance of different regularisers and cost functionals, we made use of the Berkeley segmentation dataset BSDS300 \cite{MartinFTM01}, displayed in Figure \ref{fig:bsds300}.
We resized each image to 128 pixels on its shortest edge, and take the $128\times 128$ top-left square of the image.
To this data set, we applied pixelwise Gaussian noise of variance $\sigma^2=2,10$, and $20$.
We tested the performance of both cost functionals, $\costhubertv$ and $\costltwo$, as well as the $\TGV^2$, $\ICTV$, and $\TV$ regularisers on this dataset, for all noise levels. In the first instance, reported in Figures \ref{fig:bsds300-individual-huber-2}--\ref{fig:bsds300-individual-ltwo-20} (noise levels $\sigma^2=2,20$ only), and Tables \ref{table:bsds300-individual-2}--\ref{table:bsds300-individual-20}, we applied the proposed bi-level learning model on each image individually, to learn the optimal parameters specifically for that image, and a correponding noisy image for all of the noise levels separately. For the algorithm, we use the same parametrisation as in Section \ref{sec:denoising}.

\begin{figure}
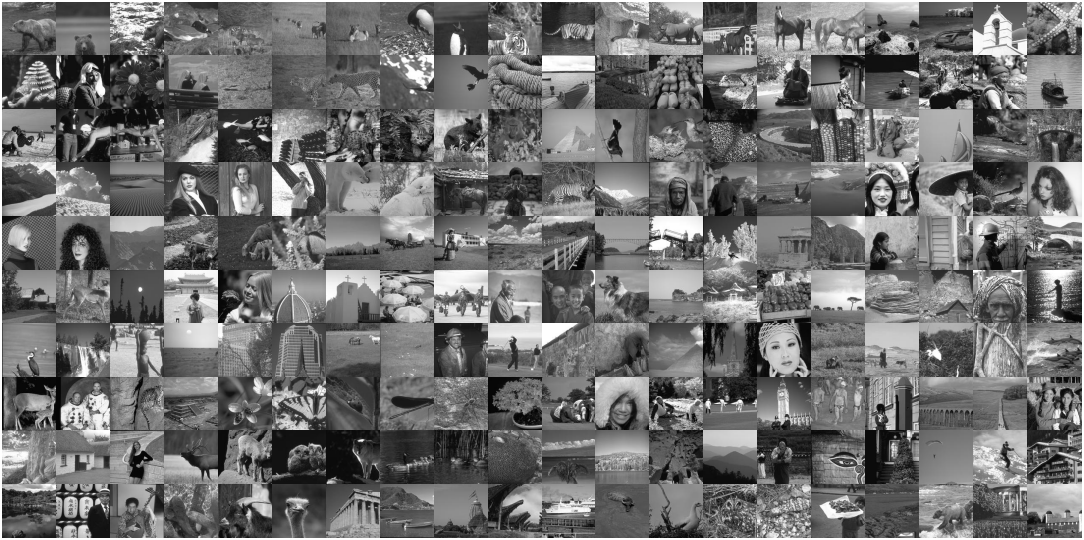

    \immatrix{img/bsds300}{orig}{20}
    \caption{The 200 images of the Berkeley segmentation dataset BSDS300 \cite{MartinFTM01}, cropped to be rectangular, keeping top left corner, and resized to $128 \times 128$.}
    \label{fig:bsds300}
\end{figure}

The figures display the noisy images, and indicate by colour coding the best result as judged by the structural similarity measure SSIM \cite{wang2004ssim}, PSNR, and the objective function value ($\costhubertv$ or $\costltwo$ cost). These criteria are, respectively, the top, middle, and bottom rows of colour-coding squares. Red square indicates that $\TV$ performed the best, green square indicates that $\ICTV$ performed the best, and blue square indicates that $\TGV^2$ performed the best---this is naturally for the optimal parameters for the corresponding regulariser and cost functional discovered by our algorithms. 

In the tables, we report the information in a more concise numerical fashion, indicating the mean, standard deviation, and median for all the different criteria (SSIM, PSNR, and cost functional value), as well as the number of images for which each regulariser performed the best. We recall that SSIM is normalised to $[0, 1]$, with higher value better. Moreover, we perform a statistical 95\% one-tailed paired t-test on each of the criteria, and a pair of regularisers, to see whether any pair of regularisers can be ordered. If so, this is indicated in the last row of each of the tables. 

Overall, studying the t-test and other data, the ordering of the regularisers appears to be
\[
    \ICTV > \TGV^2 > \TV.
\]
This is rather surprising, as in many specific examples, $\TGV^2$ has been observed to perform better than $\ICTV$, see our Figures \ref{fig:res-dataset2} and \ref{fig:res-dataset11}, as well as \cite{bredies2011tgv,benning2011higher}.
Only when the noise is high, appears $\TGV^2$ to come on par with $\ICTV$ with the $\costhubertv$ cost functional in Figure \ref{fig:bsds300-individual-huber-20} and Table \ref{table:bsds300-individual-20}.

A more detailed study of the results in Figures \ref{fig:bsds300-individual-huber-2}--\ref{fig:bsds300-individual-ltwo-20} seems to indicate that $\TGV^2$ performs better than $\ICTV$ when the image contains large smooth areas, but $\ICTV$ generally performs better for more chaotic images. This observation agrees with the results in Figures \ref{fig:res-dataset2} and \ref{fig:res-dataset11}, as well as \cite{bredies2011tgv,benning2011higher}, where the images are of the former type.

\newcommand{\matrixcapt}[3]{\caption{%
    Ordering of regularisers with individual learning, #3 cost, and noise variance $\sigma^2=#1$, on the 200 images of the BSDS300 dataset, #2. Best regulariser: red=$\TV$, green=$\ICTV$, blue=$\TGV^2$;
    top=SSIM, middle=PSNR, bottom=objective value.%
}}

\newcommand{\statcapt}[2]{\caption{%
    Regulariser performance with individual learning, $\costltwo$ and $\costhubertv$ costs and noise variance $\sigma^2=#1$; BSDS300 dataset, #2. 
}}

\begin{figure}
    \immatrixX{img/bsds300}{noisy2}{20}{resimg/comparisonmatrix-bsds300-noise10-huberonly}
    \matrixcapt{2}{resized}{$\costhubertv$}
    \label{fig:bsds300-individual-huber-2}
    \bigskip
    \immatrixX{img/bsds300}{noisy2}{20}{resimg/comparisonmatrix-bsds300-noise2-l2}
    \matrixcapt{2}{resized}{$\costltwo$}
    \label{fig:bsds300-individual-ltwo-2}
\end{figure}

\begin{figure}
    \immatrixX{img/bsds300}{noisy20}{20}{resimg/comparisonmatrix-bsds300-noise20-huberonly}
    \matrixcapt{20}{resized}{$\costhubertv$}
    \label{fig:bsds300-individual-huber-20}
    \bigskip
    \immatrixX{img/bsds300}{noisy20}{20}{resimg/comparisonmatrix-bsds300-noise20-l2}
    \matrixcapt{20}{resized}{$\costltwo$}
    \label{fig:bsds300-individual-ltwo-20}
\end{figure}

\begin{table}
    \footnotesize
    \begin{tabular}{l|r|r|r|r|r|r|r|r|r|r|r|r}
	 & \multicolumn{4}{c}{SSIM} & \multicolumn{4}{|c}{PSNR} & \multicolumn{4}{|c}{value}\\
	\hline
	 & mean & std & med & best & mean & std & med & best & mean & std & med & best\\
	\hline
	Noisy data & $0.978$ & $0.015$ & $0.981$ & 0 & $41.56$ & $0.86$ & $41.95$ & 0 & $2.9\ee^{4}$ & $3.1\ee^{2}$ & $2.9\ee^{4}$ & 0\\
	\hline
	\costname{huberonly}-\mathname{tv} & $0.988$ & $0.005$ & $0.989$ & 1 & $42.57$ & $1.10$ & $42.46$ & 5 & $2.4\ee^{4}$ & $3.7\ee^{3}$ & $2.5\ee^{4}$ & 1\\
	\costname{huberonly}-\mathname{ictv} & $0.989$ & $0.005$ & $0.990$ & 141 & $42.74$ & $1.16$ & $42.62$ & 143 & $2.3\ee^{4}$ & $3.9\ee^{3}$ & $2.4\ee^{4}$ & 137\\
	\costname{huberonly}-\mathname{tgv2} & $0.989$ & $0.005$ & $0.989$ & 58 & $42.70$ & $1.17$ & $42.55$ & 52 & $2.4\ee^{4}$ & $4.0\ee^{3}$ & $2.5\ee^{4}$ & 62\\
	\hline
	95\% t-test &\multicolumn{4}{c}{\mathname{ictv} $>$ \mathname{tgv2} $>$ \mathname{tv}} & \multicolumn{4}{|c}{\mathname{ictv} $>$ \mathname{tgv2} $>$ \mathname{tv}} & \multicolumn{4}{|c}{\mathname{ictv} $>$ \mathname{tgv2} $>$ \mathname{tv}}\\
	\hline
	\costname{l2}-\mathname{tv} & $0.988$ & $0.005$ & $0.988$ & 2 & $42.64$ & $1.14$ & $42.50$ & 2 & $0.41$ & $0.08$ & $0.43$ & 2\\
	\costname{l2}-\mathname{ictv} & $0.988$ & $0.005$ & $0.989$ & 142 & $42.79$ & $1.18$ & $42.64$ & 148 & $0.39$ & $0.08$ & $0.41$ & 148\\
	\costname{l2}-\mathname{tgv2} & $0.988$ & $0.005$ & $0.989$ & 56 & $42.76$ & $1.19$ & $42.58$ & 50 & $0.40$ & $0.08$ & $0.42$ & 50\\
	\hline
	95\% t-test &\multicolumn{4}{c}{\mathname{ictv} $>$ \mathname{tgv2} $>$ \mathname{tv}} & \multicolumn{4}{|c}{\mathname{ictv} $>$ \mathname{tgv2} $>$ \mathname{tv}} & \multicolumn{4}{|c}{\mathname{ictv} $>$ \mathname{tgv2} $>$ \mathname{tv}}\\
\end{tabular}

    \statcapt{2}{resized}
    \label{table:bsds300-individual-2}
    \bigskip
    \footnotesize
    \begin{tabular}{l|r|r|r|r|r|r|r|r|r|r|r|r}
	 & \multicolumn{4}{c}{SSIM} & \multicolumn{4}{|c}{PSNR} & \multicolumn{4}{|c}{value}\\
	\hline
	 & mean & std & med & best & mean & std & med & best & mean & std & med & best\\
	\hline
	Noisy data & $0.731$ & $0.120$ & $0.744$ & 0 & $27.72$ & $0.88$ & $28.09$ & 0 & $1.4\ee^{5}$ & $2.5\ee^{3}$ & $1.4\ee^{5}$ & 0\\
	\hline
	\costname{huberonly}-\mathname{tv} & $0.898$ & $0.036$ & $0.900$ & 4 & $31.28$ & $1.63$ & $30.97$ & 8 & $7.3\ee^{4}$ & $2.2\ee^{4}$ & $7.3\ee^{4}$ & 1\\
	\costname{huberonly}-\mathname{ictv} & $0.906$ & $0.034$ & $0.909$ & 139 & $31.54$ & $1.68$ & $31.21$ & 142 & $7.1\ee^{4}$ & $2.2\ee^{4}$ & $7.1\ee^{4}$ & 121\\
	\costname{huberonly}-\mathname{tgv2} & $0.905$ & $0.035$ & $0.907$ & 57 & $31.47$ & $1.72$ & $31.10$ & 50 & $7.1\ee^{4}$ & $2.2\ee^{4}$ & $7.1\ee^{4}$ & 78\\
	\hline
	95\% t-test &\multicolumn{4}{c}{\mathname{ictv} $>$ \mathname{tgv2} $>$ \mathname{tv}} & \multicolumn{4}{|c}{\mathname{ictv} $>$ \mathname{tgv2} $>$ \mathname{tv}} & \multicolumn{4}{|c}{\mathname{ictv} $>$ \mathname{tgv2} $>$ \mathname{tv}}\\
	\hline
	\costname{l2}-\mathname{tv} & $0.897$ & $0.033$ & $0.898$ & 9 & $31.54$ & $1.76$ & $31.15$ & 2 & $5.52$ & $1.89$ & $5.51$ & 2\\
	\costname{l2}-\mathname{ictv} & $0.903$ & $0.032$ & $0.903$ & 131 & $31.72$ & $1.76$ & $31.33$ & 148 & $5.30$ & $1.81$ & $5.35$ & 148\\
	\costname{l2}-\mathname{tgv2} & $0.902$ & $0.033$ & $0.903$ & 60 & $31.67$ & $1.80$ & $31.28$ & 50 & $5.38$ & $1.87$ & $5.39$ & 50\\
	\hline
	95\% t-test &\multicolumn{4}{c}{\mathname{ictv} $>$ \mathname{tgv2} $>$ \mathname{tv}} & \multicolumn{4}{|c}{\mathname{ictv} $>$ \mathname{tgv2} $>$ \mathname{tv}} & \multicolumn{4}{|c}{\mathname{ictv} $>$ \mathname{tgv2} $>$ \mathname{tv}}\\
\end{tabular}

    \statcapt{10}{resized}
    \label{table:bsds300-individual-10}
    \bigskip
    \footnotesize
    \begin{tabular}{l|r|r|r|r|r|r|r|r|r|r|r|r}
	 & \multicolumn{4}{c}{SSIM} & \multicolumn{4}{|c}{PSNR} & \multicolumn{4}{|c}{value}\\
	\hline
	 & mean & std & med & best & mean & std & med & best & mean & std & med & best\\
	\hline
	Noisy data & $0.505$ & $0.143$ & $0.516$ & 0 & $21.80$ & $0.92$ & $22.14$ & 0 & $2.8\ee^{5}$ & $7.9\ee^{3}$ & $2.8\ee^{5}$ & 0\\
	\hline
	\costname{huberonly}-\mathname{tv} & $0.795$ & $0.063$ & $0.799$ & 7 & $27.27$ & $1.64$ & $27.02$ & 11 & $1.0\ee^{5}$ & $3.5\ee^{4}$ & $9.7\ee^{4}$ & 1\\
	\costname{huberonly}-\mathname{ictv} & $0.810$ & $0.061$ & $0.814$ & 120 & $27.52$ & $1.66$ & $27.24$ & 125 & $9.7\ee^{4}$ & $3.4\ee^{4}$ & $9.6\ee^{4}$ & 79\\
	\costname{huberonly}-\mathname{tgv2} & $0.808$ & $0.062$ & $0.814$ & 73 & $27.50$ & $1.74$ & $27.15$ & 64 & $9.8\ee^{4}$ & $3.5\ee^{4}$ & $9.5\ee^{4}$ & 120\\
	\hline
	95\% t-test &\multicolumn{4}{c}{\mathname{ictv} $>$ \mathname{tgv2} $>$ \mathname{tv}} & \multicolumn{4}{|c}{\mathname{ictv}, \mathname{tgv2} $>$ \mathname{tv}} & \multicolumn{4}{|c}{\mathname{ictv}, \mathname{tgv2} $>$ \mathname{tv}}\\
	\hline
	\costname{l2}-\mathname{tv} & $0.802$ & $0.056$ & $0.804$ & 8 & $27.70$ & $1.93$ & $27.28$ & 0 & $13.65$ & $5.53$ & $13.14$ & 0\\
	\costname{l2}-\mathname{ictv} & $0.811$ & $0.056$ & $0.816$ & 126 & $27.86$ & $1.91$ & $27.45$ & 138 & $13.14$ & $5.22$ & $12.62$ & 138\\
	\costname{l2}-\mathname{tgv2} & $0.810$ & $0.057$ & $0.814$ & 66 & $27.83$ & $1.94$ & $27.41$ & 62 & $13.28$ & $5.38$ & $12.77$ & 62\\
	\hline
	95\% t-test &\multicolumn{4}{c}{\mathname{ictv} $>$ \mathname{tgv2} $>$ \mathname{tv}} & \multicolumn{4}{|c}{\mathname{ictv} $>$ \mathname{tgv2} $>$ \mathname{tv}} & \multicolumn{4}{|c}{\mathname{ictv} $>$ \mathname{tgv2} $>$ \mathname{tv}}\\
\end{tabular}

    \statcapt{20}{resized}
    \label{table:bsds300-individual-20}
\end{table}

One possible reason for the better performance of $\ICTV$ could be that $\TGV^2$ has more degrees of freedom---in $\ICTV$ we essentially constrain $w=\grad v$ for some function $v$---and therefore overfits to the noisy data, until the noise level becomes so high that overfitting would become too high for any parameter.
To see whether this is true, we also performed batch learning, learning a single set of parameters for all images with the same noise level. That is, we studied the model
\[
    \min_{\alphavec} \sum_{i=1}^N F_i(u_{i,\alphavec})
    \quad
    \text{s.t.}\quad
    u_{i,\alphavec} \in \argmin_{u\in H^1(\Omega)} \frac{1}{2}\norm{f_i-u}_{L^2(\Omega)}^2
        + R_{\alphavec}^{\gamma,\mu}(u),
\]
with
\[
    F_i(u)=\frac{1}{2}\norm{f_{0,i}-u}^2_{L^2(\Omega)},
    \quad\text{or}\quad
    F_i(u)=\int_\Omega \abs{\grad(f_{0,i}-u)}_\gamma \d x,
\]
where $\vec \alpha=(\alpha, \beta)$, $f_1,\ldots,f_N$ are the $N=200$ noisy images with the same noise level, and $f_{0,1},\ldots,f_{0,N}$ the original noise free images.

The results are in Figures \ref{fig:bsds300-batch-huber-2}--\ref{fig:bsds300-batch-ltwo-20} (noise levels $\sigma^2=2,20$ only), and Tables \ref{table:bsds300-batch-2}--\ref{table:bsds300-batch-20}. The results are still roughly the same as with individual learning. Again, only with high noise in Table \ref{table:bsds300-batch-20}, does $\TGV^2$ not lose to $\ICTV$.
Another interesting observation is that $\TV$ starts to be frequently the best regulariser for individual images, although still statistically does worse than either $\ICTV$ or $\TGV^2$.

For the first image of the data set, $\ICTV$ does in all of the Figures \ref{fig:bsds300-individual-huber-2}--\ref{fig:bsds300-batch-ltwo-20} better than $\TGV^2$, while for the second image, the situation is reversed. We have highlighted these two images for the $\costhubertv$ cost in Figures \ref{fig:ictv-better-2}--\ref{fig:tgv-better-20}, for both noise levels $\sigma=2$ and $\sigma=20$. In the case where $\ICTV$ does better, hardly any difference can be observed by the eye, while for second image $\TGV^2$ clearly has less stair-casing in the smooth areas of the image, especially with the noise level $\sigma=20$.

Based on this study, it therefore seems that $\ICTV$ is the most reliable regulariser of the ones tested, when the type of image being processed is unknown, and low SSIM, PSNR or $\costhubertv$ cost functional value is desired. But as can be observed for individual images, it can within large smooth areas exhibit artefacts that are avoided by the use of $\TGV^2$.


\newcommand{\matrixcaptsim}[3]{\caption{%
    Ordering of regularisers with batch learning, #3 cost, and noise variance $\sigma^2=#1$, on the 200 images of the BSDS300 dataset, #2. Best regulariser: red=$\TV$, green=$\ICTV$, blue=$\TGV^2$;
    top=SSIM, middle=PSNR, bottom=objective value.%
}}

\newcommand{\statcaptsim}[2]{\caption{%
    Regulariser performance with batch learning, $\costhubertv$ and $\costltwo$ costs, noise variance $\sigma^2=#1$; BSDS300 dataset, #2. 
}}

\begin{figure}
    \immatrixX{img/bsds300}{noisy2}{20}{resimg/comparisonmatrix-bsds300-ALL-noise10-huberonly}
    \matrixcaptsim{2}{resized}{$\costhubertv$}
    \label{fig:bsds300-batch-huber-2}
    \bigskip
    \immatrixX{img/bsds300}{noisy2}{20}{resimg/comparisonmatrix-bsds300-ALL-noise2-l2}
    \matrixcaptsim{2}{resized}{$\costltwo$}
    \label{fig:bsds300-batch-ltwo-2}
\end{figure}

\begin{figure}
    \immatrixX{img/bsds300}{noisy20}{20}{resimg/comparisonmatrix-bsds300-ALL-noise20-huberonly}
    \matrixcaptsim{20}{resized}{$\costhubertv$}
    \label{fig:bsds300-batch-huber-20}
    \bigskip
    \immatrixX{img/bsds300}{noisy20}{20}{resimg/comparisonmatrix-bsds300-ALL-noise20-l2}
    \matrixcaptsim{20}{resized}{$\costltwo$}
    \label{fig:bsds300-batch-ltwo-20}
\end{figure}

\begin{table}
    \footnotesize
    \begin{tabular}{l|r|r|r|r|r|r|r|r|r|r|r|r}
	 & \multicolumn{4}{c}{SSIM} & \multicolumn{4}{|c}{PSNR} & \multicolumn{4}{|c}{value}\\
	\hline
	 & mean & std & med & best & mean & std & med & best & mean & std & med & best\\
	\hline
	Noisy data & $0.978$ & $0.015$ & $0.981$ & 16 & $41.56$ & $0.86$ & $41.95$ & 24 & $2.9\ee^{4}$ & $3.1\ee^{2}$ & $2.9\ee^{4}$ & 16\\
	\hline
	\costname{huberonly}-\mathname{tv} & $0.987$ & $0.006$ & $0.988$ & 23 & $42.43$ & $1.07$ & $42.37$ & 21 & $2.5\ee^{4}$ & $3.4\ee^{3}$ & $2.5\ee^{4}$ & 20\\
	\costname{huberonly}-\mathname{ictv} & $0.988$ & $0.006$ & $0.989$ & 119 & $42.56$ & $1.06$ & $42.51$ & 135 & $2.4\ee^{4}$ & $3.5\ee^{3}$ & $2.5\ee^{4}$ & 113\\
	\costname{huberonly}-\mathname{tgv2} & $0.987$ & $0.006$ & $0.989$ & 42 & $42.51$ & $1.09$ & $42.44$ & 20 & $2.4\ee^{4}$ & $3.6\ee^{3}$ & $2.5\ee^{4}$ & 51\\
	\hline
	95\% t-test &\multicolumn{4}{c}{\mathname{ictv} $>$ \mathname{tgv2} $>$ \mathname{tv}} & \multicolumn{4}{|c}{\mathname{ictv} $>$ \mathname{tgv2} $>$ \mathname{tv}} & \multicolumn{4}{|c}{\mathname{ictv} $>$ \mathname{tgv2} $>$ \mathname{tv}}\\
	\hline
	\costname{l2}-\mathname{tv} & $0.986$ & $0.007$ & $0.987$ & 13 & $42.46$ & $0.95$ & $42.43$ & 17 & $0.42$ & $0.07$ & $0.43$ & 17\\
	\costname{l2}-\mathname{ictv} & $0.987$ & $0.007$ & $0.988$ & 139 & $42.57$ & $0.95$ & $42.56$ & 128 & $0.41$ & $0.07$ & $0.42$ & 128\\
	\costname{l2}-\mathname{tgv2} & $0.987$ & $0.007$ & $0.988$ & 38 & $42.53$ & $0.97$ & $42.51$ & 40 & $0.41$ & $0.07$ & $0.42$ & 40\\
	\hline
	95\% t-test &\multicolumn{4}{c}{\mathname{ictv} $>$ \mathname{tgv2} $>$ \mathname{tv}} & \multicolumn{4}{|c}{\mathname{ictv} $>$ \mathname{tgv2} $>$ \mathname{tv}} & \multicolumn{4}{|c}{\mathname{ictv} $>$ \mathname{tgv2} $>$ \mathname{tv}}\\
\end{tabular}

    \statcaptsim{2}{resized}
    \label{table:bsds300-batch-2}
    \bigskip
    \footnotesize
    \begin{tabular}{l|r|r|r|r|r|r|r|r|r|r|r|r}
	 & \multicolumn{4}{c}{SSIM} & \multicolumn{4}{|c}{PSNR} & \multicolumn{4}{|c}{value}\\
	\hline
	 & mean & std & med & best & mean & std & med & best & mean & std & med & best\\
	\hline
	Noisy data & $0.731$ & $0.120$ & $0.744$ & 8 & $27.72$ & $0.88$ & $28.09$ & 2 & $1.4\ee^{5}$ & $2.5\ee^{3}$ & $1.4\ee^{5}$ & 0\\
	\hline
	\costname{huberonly}-\mathname{tv} & $0.893$ & $0.035$ & $0.897$ & 23 & $31.24$ & $1.87$ & $30.94$ & 23 & $7.5\ee^{4}$ & $2.2\ee^{4}$ & $7.3\ee^{4}$ & 18\\
	\costname{huberonly}-\mathname{ictv} & $0.897$ & $0.034$ & $0.902$ & 134 & $31.36$ & $1.81$ & $31.11$ & 150 & $7.4\ee^{4}$ & $2.2\ee^{4}$ & $7.2\ee^{4}$ & 107\\
	\costname{huberonly}-\mathname{tgv2} & $0.896$ & $0.035$ & $0.901$ & 35 & $31.31$ & $1.88$ & $31.01$ & 25 & $7.4\ee^{4}$ & $2.3\ee^{4}$ & $7.2\ee^{4}$ & 75\\
	\hline
	95\% t-test &\multicolumn{4}{c}{\mathname{ictv} $>$ \mathname{tgv2} $>$ \mathname{tv}} & \multicolumn{4}{|c}{\mathname{ictv} $>$ \mathname{tgv2} $>$ \mathname{tv}} & \multicolumn{4}{|c}{\mathname{ictv}, \mathname{tgv2} $>$ \mathname{tv}}\\
	\hline
	\costname{l2}-\mathname{tv} & $0.887$ & $0.035$ & $0.889$ & 29 & $31.31$ & $1.50$ & $31.15$ & 25 & $5.72$ & $1.91$ & $5.51$ & 25\\
	\costname{l2}-\mathname{ictv} & $0.889$ & $0.036$ & $0.893$ & 127 & $31.41$ & $1.44$ & $31.28$ & 131 & $5.57$ & $1.83$ & $5.37$ & 131\\
	\costname{l2}-\mathname{tgv2} & $0.888$ & $0.035$ & $0.891$ & 44 & $31.38$ & $1.50$ & $31.20$ & 44 & $5.64$ & $1.90$ & $5.44$ & 44\\
	\hline
	95\% t-test &\multicolumn{4}{c}{\mathname{ictv} $>$ \mathname{tgv2} $>$ \mathname{tv}} & \multicolumn{4}{|c}{\mathname{ictv} $>$ \mathname{tgv2} $>$ \mathname{tv}} & \multicolumn{4}{|c}{\mathname{ictv} $>$ \mathname{tgv2} $>$ \mathname{tv}}\\
\end{tabular}

    \statcaptsim{10}{resized}
    \label{table:bsds300-batch-10}
    \bigskip
    \footnotesize
    \begin{tabular}{l|r|r|r|r|r|r|r|r|r|r|r|r}
	 & \multicolumn{4}{c}{SSIM} & \multicolumn{4}{|c}{PSNR} & \multicolumn{4}{|c}{value}\\
	\hline
	 & mean & std & med & best & mean & std & med & best & mean & std & med & best\\
	\hline
	Noisy data & $0.505$ & $0.143$ & $0.516$ & 4 & $21.80$ & $0.92$ & $22.14$ & 1 & $2.8\ee^{5}$ & $7.9\ee^{3}$ & $2.8\ee^{5}$ & 0\\
	\hline
	\costname{huberonly}-\mathname{tv} & $0.789$ & $0.067$ & $0.798$ & 18 & $27.37$ & $2.13$ & $26.98$ & 24 & $1.0\ee^{5}$ & $3.7\ee^{4}$ & $9.8\ee^{4}$ & 14\\
	\costname{huberonly}-\mathname{ictv} & $0.795$ & $0.065$ & $0.804$ & 139 & $27.46$ & $2.10$ & $27.05$ & 141 & $1.0\ee^{5}$ & $3.6\ee^{4}$ & $9.6\ee^{4}$ & 91\\
	\costname{huberonly}-\mathname{tgv2} & $0.794$ & $0.066$ & $0.804$ & 39 & $27.44$ & $2.12$ & $27.04$ & 34 & $1.0\ee^{5}$ & $3.7\ee^{4}$ & $9.6\ee^{4}$ & 95\\
	\hline
	95\% t-test &\multicolumn{4}{c}{\mathname{ictv} $>$ \mathname{tgv2} $>$ \mathname{tv}} & \multicolumn{4}{|c}{\mathname{ictv} $>$ \mathname{tgv2} $>$ \mathname{tv}} & \multicolumn{4}{|c}{\mathname{tgv2} $>$ \mathname{ictv} $>$ \mathname{tv}}\\
	\hline
	\costname{l2}-\mathname{tv} & $0.786$ & $0.053$ & $0.790$ & 31 & $27.50$ & $1.71$ & $27.27$ & 33 & $14.11$ & $5.78$ & $13.16$ & 33\\
	\costname{l2}-\mathname{ictv} & $0.790$ & $0.054$ & $0.790$ & 123 & $27.56$ & $1.64$ & $27.37$ & 119 & $13.84$ & $5.54$ & $12.75$ & 119\\
	\costname{l2}-\mathname{tgv2} & $0.789$ & $0.053$ & $0.793$ & 46 & $27.55$ & $1.70$ & $27.33$ & 48 & $13.93$ & $5.73$ & $12.95$ & 48\\
	\hline
	95\% t-test &\multicolumn{4}{c}{\mathname{ictv}, \mathname{tgv2} $>$ \mathname{tv}} & \multicolumn{4}{|c}{\mathname{ictv}, \mathname{tgv2} $>$ \mathname{tv}} & \multicolumn{4}{|c}{\mathname{ictv} $>$ \mathname{tgv2} $>$ \mathname{tv}}\\
\end{tabular}

    \statcaptsim{20}{resized}
    \label{table:bsds300-batch-20}
\end{table}

\def\EXictv{100075}
\def\EXtgv{100080}

\newcommand{\sub}[3]{%
    \begin{subfigure}[t]{0.32\textwidth}%
        \includegraphics[width=\textwidth]{{img/bsds300/#2_#3}.jpg}%
        \subcaption{#1}%
    \end{subfigure}%
}

\nprounddigits{2}
\npdecimalsign{.}

\newcommand{\subx}[5]{%
    \begin{subfigure}[t]{0.32\textwidth}%
        \input{resimg/#1-#2-#5#3noise#4-linesearch1-bfgs1-beta0-uinitreg_0.200000_-vals.tex}
        \includegraphics[width=\textwidth]{{resimg/#1-#2-#5#3noise#4-linesearch1-bfgs1-beta0-uinitreg_0.200000_}.png}%
        \def\a{dataset}%
        \def\b{#5}%
        \caption{\ifx\a\b Individual\else Batch\fi~\costname{#2}-\mathname{#1}, 
        \\PSNR=\numprint{\RESpsnr}, SSIM=\numprint{\RESssim}}%
    \end{subfigure}%
}

\begin{figure}
    \centering
    \sub{Original}{orig}{\EXictv}
    \input{resimg/tgv2-huberonly-dataset\EXictvnoise2-linesearch1-bfgs1-beta0-uinitreg_0.200000_-vals.tex}
    \input{resimg/tgv2-huberonly-bsds300-ALL\EXictvnoise2-linesearch1-bfgs1-beta0-uinitreg_0.200000_-vals.tex}
    \sub{Noisy, $\sigma=2$}{noisy2}{\EXictv}
    \input{resimg/ictv-huberonly-dataset\EXictvnoise2-linesearch1-bfgs1-beta0-uinitreg_0.200000_-vals.tex}
    \input{resimg/ictv-huberonly-bsds300-ALL\EXictvnoise2-linesearch1-bfgs1-beta0-uinitreg_0.200000_-vals.tex}
    \caption{Image for which $\ICTV$ performs better than $\TGV^2$, $\sigma=2$}
    \label{fig:ictv-better-2}
\end{figure}

\begin{figure}
    \centering
    \sub{Original}{orig}{\EXictv}
    \input{resimg/tgv2-huberonly-dataset\EXictvnoise20-linesearch1-bfgs1-beta0-uinitreg_0.200000_-vals.tex}
    \input{resimg/tgv2-huberonly-bsds300-ALL\EXictvnoise20-linesearch1-bfgs1-beta0-uinitreg_0.200000_-vals.tex}
    \sub{Noisy, $\sigma=20$}{noisy20}{\EXictv}
    \input{resimg/ictv-huberonly-dataset\EXictvnoise20-linesearch1-bfgs1-beta0-uinitreg_0.200000_-vals.tex}
    \input{resimg/ictv-huberonly-bsds300-ALL\EXictvnoise20-linesearch1-bfgs1-beta0-uinitreg_0.200000_-vals.tex}
    \caption{Image for which $\ICTV$ performs better than $\TGV^2$, $\sigma=20$}
    \label{fig:ictv-better-20}
\end{figure}

\begin{figure}
    \centering
    \sub{Original}{orig}{\EXtgv}
    \input{resimg/tgv2-huberonly-dataset\EXtgvnoise2-linesearch1-bfgs1-beta0-uinitreg_0.200000_-vals.tex}
    \input{resimg/tgv2-huberonly-bsds300-ALL\EXtgvnoise2-linesearch1-bfgs1-beta0-uinitreg_0.200000_-vals.tex}
    \sub{Noisy, $\sigma=2$}{noisy2}{\EXtgv}
    \input{resimg/ictv-huberonly-dataset\EXtgvnoise2-linesearch1-bfgs1-beta0-uinitreg_0.200000_-vals.tex}
    \input{resimg/ictv-huberonly-bsds300-ALL\EXtgvnoise2-linesearch1-bfgs1-beta0-uinitreg_0.200000_-vals.tex}
    \caption{Image for which $\TGV^2$ performs better than $\ICTV$, $\sigma=2$}
    \label{fig:tgv-better-2}
\end{figure}

\begin{figure}
    \centering
    \sub{Original}{orig}{\EXtgv}
    \input{resimg/tgv2-huberonly-dataset\EXtgvnoise20-linesearch1-bfgs1-beta0-uinitreg_0.200000_-vals.tex}
    \input{resimg/tgv2-huberonly-bsds300-ALL\EXtgvnoise20-linesearch1-bfgs1-beta0-uinitreg_0.200000_-vals.tex}
    \sub{Noisy, $\sigma=20$}{noisy20}{\EXtgv}
    \input{resimg/ictv-huberonly-dataset\EXtgvnoise20-linesearch1-bfgs1-beta0-uinitreg_0.200000_-vals.tex}
    \input{resimg/ictv-huberonly-bsds300-ALL\EXtgvnoise20-linesearch1-bfgs1-beta0-uinitreg_0.200000_-vals.tex}
    \caption{Image for which $\TGV^2$ performs better than $\ICTV$, $\sigma=20$}
    \label{fig:tgv-better-20}
\end{figure}

\subsection{The choice of cost functional}

The $\costltwo$ cost functional naturally obtains better PSNR than $\costhubertv$, as the two former are equivalent. Comparing the results for the two cost funtionals in Tables \ref{table:bsds300-individual-2}--\ref{table:bsds300-individual-20}, we may however observe that for low noise levels $\sigma^2=2,10$, and generally for batch learning, $\costhubertv$ attains better (higher) SSIM.
Since SSIM better captures \cite{wang2004ssim} the visual quality of images than PSNR, this recommends the use of our novel total variation cost functional $\costhubertv$. Of course, one might attempt to optimise the SSIM. This is however a non-convex functional, which will pose additional numerical challenges avoided by the convex total variation cost.

\section*{Conclusion and Outlook} In this paper we propose a bilevel optimisation method in function space for learning the optimal choice of parameters in higher-order total variation regularisation. We present a rigorous analysis of this optimisation problem as well as a numerical discussion in the context of image denoising. In particular, we make use of the bilevel learning approach to compare the performance -- in terms of returned image quality -- of TV, ICTV and TGV regularisation. A statistical analysis, carried out on a dataset of 200 images, suggest that ICTV performs slightly better than TGV, and both perform better than TV, in average. For denoising of images with a high noise level ICTV and TGV score comparably well. For images with large smooth areas TGV performs better than ICTV.

Moreover, we propose a new cost functional for the bilevel learning problem, which exhibits interesting theoretical properties and has a better behaviour with respect to the PSNR related L$^2$ cost used previously in the literature. This study raises the question of other, alternative cost functionals. For instance, one could be tempted to used the SSIM as cost, but its non-convexity might present several analytical and numerical difficulties. The new cost functional, proposed in this paper, turns out to be a good compromise between image quality measure and analytically tractable cost term. 

\section*{Acknowledgements}
This project has been supported by King Abdullah University of Science and Technology (KAUST) Award No. KUK-I1-007-43, EPSRC grants Nr. EP/J009539/1 and Nr. EP/M00483X/1, the Escuela Polit\'ecnica Nacional de Quito under award PIS 12-14 and the MATHAmSud project SOCDE `Sparse Optimal Control of Differential Equations'. While in Quito, T. Valkonen has moreover been supported by a Prometeo scholarship of the Senescyt (Ecuadorian Ministry of Science, Technology, Education, and Innovation).

{\small
\bibliographystyle{plain}
\bibliography{abbrevs,bib,bib-own}
}
\end{document}